\documentclass[11pt,reqno]{amsart}
\usepackage{amssymb,latexsym,graphicx,amscd}
\usepackage{lscape}
\usepackage[all]{xy}
\usepackage{color}
\usepackage{epic,eepic}
\usepackage{xspace}
\usepackage{mathrsfs}
\usepackage{setspace}
\usepackage{cases}
\usepackage{bbm}         
\usepackage{url}
\newtheorem{Thm}{Theorem}[section]
\newtheorem{Lem}[Thm]{Lemma}
\newtheorem{Cor}[Thm]{Corollary}
\newtheorem{Prop}[Thm]{Proposition}

\newtheorem{Def}[Thm]{Definition}

\theoremstyle{definition}

\newcommand{\CC}{\mathbb{C}}

\newcommand{\PP}{\mathbb{P}}
\newcommand{\Z}{\mathbb{Z}}

\newcommand{\N}{\mathbb{N}}
\newcommand{\C}{\mathbb{C}}

\newcommand{\df}{\colon}

\newcommand{\cA}{{\mathcal A}}

\newcommand{\cF}{{\mathcal F}}

\newcommand{\cM}{{\mathcal M}}

\newcommand{\cP}{{\mathcal P}}

\newcommand{\g}{\mathfrak{g}}
\newcommand{\n}{\mathfrak{n}}
\newcommand{\h}{\mathfrak{h}}

\newcommand{\ba}{\mathbf{a}}
\newcommand{\bb}{\mathbf{b}}

\newcommand{\bd}{\mathbf{d}}
\newcommand{\be}{\mathbf{e}}
\newcommand{\bi}{{\mathbf i}}
\newcommand{\bj}{{\mathbf j}}

\newcommand{\bm}{{\mathbf m}}

\newcommand{\br}{{\mathbf r}}
\newcommand{\bs}{{\mathbf s}}
\newcommand{\bt}{{\mathbf t}}

\newcommand{\bx}{{\mathbf x}}

\newcommand{\bz}{{\mathbf z}}

\newcommand{\vep}{\varepsilon}
\newcommand{\la}{\lambda}
\newcommand{\de}{\delta}

\newcommand{\rk}{\operatorname{rank}}

\newcommand{\mx}{{\rm max}}

\newcommand{\rep}{\operatorname{rep}}

\newcommand{\pdim}{\operatorname{proj.dim}}
\newcommand{\idim}{\operatorname{inj.dim}}

\newcommand{\Hom}{\operatorname{Hom}}
\newcommand{\Ext}{\operatorname{Ext}}

\newcommand{\Fl}{\operatorname{Fl}}

\newcommand{\bil}[1]{\langle #1\rangle}

\newcommand{\Irr}{\operatorname{Irr}}
\newcommand{\supp}{\operatorname{supp}}

\newcommand{\bsm}{\begin{smallmatrix}}
\newcommand{\esm}{\end{smallmatrix}}

\newcommand{\bbsm}{\left[\begin{smallmatrix}}
\newcommand{\besm}{\end{smallmatrix}\right]}

\newcommand{\bbm}{\begin{matrix}}
\newcommand{\ebm}{\end{matrix}}

\newcommand{\diag}{{\rm diag}}

\newcommand{\GL}{\operatorname{GL}}

\newcommand{\vp}{{\rm l.f.}}

\def\lf{\mathrm{l.f.}}
\def\<{\langle\,}
\def\>{\,\rangle}
\newcommand{\rkv}{\underline{\rk}}
\def\Gr{{\rm Gr}}
\def\Si{\Sigma}
\def\AA{\mathcal{A}}

\begin{document}

\title[Quivers with relations for symmetrizable Cartan matrices V]
{Quivers with relations for symmetrizable Cartan matrices V:
Caldero-Chapoton formulas}

\author{Christof Gei{\ss}}
\address{Christof Gei{\ss}\newline
Instituto de Matem\'aticas\newline
Universidad Nacional Aut{\'o}noma de M{\'e}xico\newline
Ciudad Universitaria\newline
04510 M{\'e}xico D.F.\newline
M{\'e}xico}
\email{christof.geiss@im.unam.mx}
\thanks{The first named author acknowledges partial support from 
CoNaCyT grant no. 239255.}

\author{Bernard Leclerc}
\address{Bernard Leclerc\newline
Normandie Univ, Unicaen\newline
CNRS, LMNO, UMR 6139\newline
F-14032 Caen Cedex\newline
France}
\email{bernard.leclerc@unicaen.fr}

\author{Jan Schr\"oer}
\address{Jan Schr\"oer\newline
Mathematisches Institut\newline
Universit\"at Bonn\newline
Endenicher Allee 60\newline
53115 Bonn\newline
Germany}
\email{schroer@math.uni-bonn.de}



\begin{abstract}
We generalize the Caldero-Chapoton formula for cluster algebras of finite type 
to the skew-symmetrizable case. This is done by replacing representation 
categories of Dynkin quivers by categories of locally free modules over certain
Iwanaga-Gorenstein algebras introduced in~\cite{GLS1}.
The proof relies on the realization of the positive part of the enveloping
algebra of a simple Lie algebra of the same finite type as a convolution 
algebra of constructible functions on representation varieties of $H$, 
given in~\cite{GLS3}. Along the way, we obtain a new result on the PBW basis 
of this convolution algebra. 
\end{abstract}

\maketitle

\setcounter{tocdepth}{1}
\numberwithin{equation}{section}
\tableofcontents

\parskip2mm


\section{Introduction and main results}\label{sec:intro}


Fomin and Zelevinsky \cite{FZ2} have shown that the classification of cluster 
algebras of finite type is identical to the Cartan-Killing classification of 
semisimple Lie algebras and finite root systems.
For cluster algebras of type $A$, $D$, $E$, Caldero and Chapoton~\cite{CC} 
proved a geometric formula for the cluster expansions of cluster variables 
with respect to an acyclic initial seed. If $Q$ is the quiver (of the same 
type $A$, $D$, $E$) attached to this initial seed, the coefficients of the 
cluster expansion are equal to the Euler characteristics of the Grassmannians 
of sub-representations of the indecomposable representations of $Q$. In this 
paper, we generalize the Caldero-Chapoton formula to all finite types by
replacing representation categories of Dynkin quivers by categories of locally 
free modules over certain
Iwanaga-Gorenstein algebras introduced in \cite{GLS1}.

Let us state our result more precisely.
Let $H = H(C,D,\Omega)$ be the Iwanaga-Gorenstein algebra introduced 
in~\cite[Section 1.4]{GLS1}.
(The definition of $H$ is recalled in Section~\ref{sec:recallpreproj}.) 
Here $C=(c_{ij})$ is an $n\times n$ Cartan matrix, $D=\diag(c_i)$ a symmetrizer 
for $C$ (that is, $DC$ is a symmetric positive definite matrix), 
and $\Omega$ an acyclic orientation.
We take as base field the field $\C$ of complex numbers. 

Let $\rep_\lf(H)\subseteq \rep(H)$ be the exact subcategory of locally free 
$H$-modules \cite[Section 1.5]{GLS1}.
We denote by $\<\cdot,\cdot\>_H$ its homological bilinear form, defined 
in~\cite[Section 4]{GLS1}.
For $M,N\in\rep_\lf(H)$ the integer $\<M,N\>_H$ depends only on the rank 
vectors $\rkv(M)$
and $\rkv(N)$. We denote by $E_i$ the indecomposable locally free $H$-module 
with rank vector
$\rkv(E_i) = (\de_{ij}\mid 1\le j\le n)$.

By \cite[Theorem 1.3]{GLS1}, the map $M \mapsto \rkv(M)$ induces a bijection 
between the set of 
isomorphism classes of indecomposable locally free rigid $H$-modules and the 
set $\Delta^+(C)$ of positive roots for $C$.
We denote by $M(\beta)\in \rep_\lf(H)$ the indecomposable rigid module with 
rank vector the positive root $\beta$.

Given $M\in\rep_\lf(H)$ and $\br\in \N^n$, let $\Gr_\lf(\br,M)$ be the 
quasi-projective variety
of locally free submodules $N$ of $M$ with rank vector $\rkv(N) = \br$.
We denote by $\chi(\Gr_\lf(\br,M))$ its Euler characteristic.

Let $B=(b_{ij})$ be the $n\times n$ skew-symmetrizable matrix defined by
\[
 b_{ij} = 
 \left\{
 \begin{matrix}
c_{ij} & \mbox{if $(j,i)\in\Omega$,}\\ 
-c_{ij} & \mbox{if $(i,j)\in\Omega$,}\\
0 & \mbox{otherwise.}
 \end{matrix}
\right.
\]
Let $R = \Z[u_1^{\pm1},\ldots,u_n^{\pm1}]$ be the Laurent polynomial ring in $n$ 
indeterminates $u_1,\ldots,u_n$. 
Let $\AA\subset R$ be the coefficient-free cluster algebra with initial seed 
$\Si = (B,(u_1,\ldots,u_n))$. (Note that $\Si$ is acyclic.)
By \cite{FZ2} this is a skew-symmetrizable cluster algebra of the same Cartan 
type as $C$,
whose cluster variables $x(\beta)$ (other than $u_1,\ldots, u_n$) are  
naturally labelled by the positive roots $\beta \in \Delta^+(C)$.

Inspired by \cite[Section 3.1]{CC} we make the following definition.
\begin{Def}\label{def-1.1}
For $M\in\rep_\lf(H)$  set 
\[
 X_M = \sum_{\br\in\N^n} \chi(\Gr_\lf(\br,M)) \prod_{i=1}^n v_i^{-\<\br,\,\rkv(E_i)\>_H-\<\rkv(E_i),\,\rkv(M)-\br\>_H},
\]
where $v_i := u_i^{1/c_i}$.
\end{Def}
It is easy to check that $X_M$ is a Laurent polynomial in the variables $u_i$ 
(see below Lemma~\ref{Lem-5.1}).
The aim of this paper is to prove the following theorem.

\begin{Thm}\label{Thm1}
 \begin{itemize}
  \item[(a)] For every locally free module $M$ we have $X_M \in \AA$.
  \item[(b)] For every locally free modules $M$ and $N$, we have 
  $X_M\cdot X_N = X_{M\oplus N}$.
  \item[(c)] The map $M\mapsto X_M$ induces a bijection between isomorphism 
  classes of locally free indecomposable 
  rigid $H$-modules and cluster variables of $\AA$ other than $u_1,\ldots,u_n$.
  More precisely, for every $\beta\in\Delta^+(C)$ there holds: 
  \[
   X_{M(\beta)} = x(\beta).
  \]
  \item[(d)] The map $M\mapsto X_M$ induces a bijection between isomorphism 
classes of locally free   rigid $H$-modules and cluster monomials of $\AA$ 
which do not contain the cluster variables $u_1,\ldots,u_n$.
 \end{itemize}
\end{Thm}

Note that when $C$ is symmetric and $D=I_n$ the identity matrix, the algebra 
$H$ is isomorphic to the path algebra
over $\C$ of the Dynkin quiver given by $C$ and the orientation $\Omega$. 
In this case all $H$-modules are locally free, and we have $v_i = u_i$. 
Moreover, indecomposable modules over this path algebra
are always rigid. Thus, one sees that Theorem~\ref{Thm1}(c) then reduces
to the classical Caldero-Chapoton formula \cite[Theorem 3.4]{CC}. 

However, it does not seem possible to extend the proof of \cite{CC}
to our setup.
We will instead rely on the results of \cite{GLS3} and \cite{YZ}. 
Here is a sketch of our proof.
By \cite{FZ4}, cluster expansions of cluster variables are completely 
determined by their $g$-vectors and their $F$-polynomials.
In \cite{YZ} a description of the $g$-vectors and $F$-polynomials of the 
cluster variables of $\AA$ (with 
respect to its initial cluster $(u_1,\ldots,u_n)$) is given in terms of the 
simply-connected complex algebraic group $G$ with
the same Cartan type as $\AA$. Let $N$ be a maximal unipotent subgroup of $G$, 
with Lie algebra $\n$.
In \cite{GLS3} we have given a geometric construction of the enveloping 
algebra $U(\n)$ as a convolution algebra $\cM(H)$ of 
constructible functions on representation varieties of locally free 
$H$-modules. In particular
for every positive root $\beta$ we obtained a primitive element $\theta_\beta$ 
of $\cM(H)$ such that 
$\theta_\beta(M(\beta)) = 1$ \cite[Theorem 6.1]{GLS3}. 
Write $\Delta^+(C) = \{\beta_1,\ldots,\beta_r\}$.
The normalized ordered products 
\[
 \theta_\ba := \theta_{\beta_1}^{(a_1)} * \cdots * \theta_{\beta_r}^{(a_r)},\qquad (\ba \in \N^r)
\]
then form a Poincar\'e-Birkhoff-Witt basis of $\cM(H)$. 
Putting 
\[
M_\ba := M(\beta_1)^{a_1} \oplus \cdots \oplus M(\beta_r)^{a_r},\qquad (\ba \in \N^r),
\]
we then prove: 
\begin{Thm}\label{Thm2}
Let the elements of $\Delta^+(C)$ be numbered so that 
$\Hom_H(M(\beta_i),M(\beta_j))= 0$ if $i<j$.
(This is possible in view of \cite[Corollary 5.8]{GLS3}.) 
For every  $\ba, \bb \in \N^r$, we have 
\[
 \theta_\ba(M_\bb) = \delta_{\ba,\bb}.
\]
\end{Thm}
Note that in contrast to the classical quiver theory, the support of the 
constructible function $\theta_\ba$ is
not in general reduced to the orbit of $M_\ba$ in its representation variety.
So Theorem~\ref{Thm2} is not at all obvious. We prove it by establishing the 
existence of certain filtrations of $M(\beta)$ (Theorem~\ref{thm:filt7}).

Let $\cM(H)^*_{\rm gr}$ denote the graded dual of $\cM(H)$, and for a locally 
free $H$-module $M$,
denote by $\delta_M \in \cM(H)^*_{\rm gr}$ the linear form given by evaluation 
at $M$. 
By Theorem~\ref{Thm2} the basis of $\cM(H)^*_{\rm gr}$ dual to the 
PBW-basis consists of all (commutative) monomials in the linear forms 
$\delta_{M(\beta)}$. 
To prove Theorem~\ref{Thm1}, we introduce for every locally free $H$-module $M$ 
its $F$-polynomial
\begin{equation}\label{eq1:1}
 F_M(t_1,\ldots,t_n) := \sum_{\br\in\N^n} \chi(\Gr_\lf(\br,M)) t_1^{r_1}\cdots t_n^{r_n}, 
\end{equation}
and we give another expression of $F_M$ in terms of $\de_M$ 
(Proposition~\ref{Prop4-2}). 
We then show using Theorem~\ref{Thm2} that $\delta_{M(\beta)}$ is equal to the 
restriction to $N$ of a generalized minor 
of $G$, and hence obtain an expression of $F_{M(\beta)}$ as the evaluation of 
this minor on a certain product of one-parameter subgroups of $N$ 
(Corollary~\ref{cor-F-pol}). 
This expression of $F_{M(\beta)}$ is very similar
to the expression of the $F$-polynomial of $x(\beta)$ given in \cite{YZ}. 
Using some commutation
relations in $G$ together with classical properties of the generalized minors, 
we can show that $F_{M(\beta)}$
is in fact equal to the $F$-polynomial of $x(\beta)$ (Theorem~\ref{Thm:11.3}). 
This is the main step in the proof of Theorem~\ref{Thm1}.  

We note that our proof also provides a representation-theoretic interpretation 
of $g$-vectors.
Indeed, for a locally free $H$-module $M$ with injective co\-resolution 
\[
 0 \to M \to \bigoplus_{1\le k\le n} I_k^{a_k} \to \bigoplus_{1\le k\le n} I_k^{b_k} \to 0,
\]
where $I_k$ denotes the injective hull of $E_k$, define 
\begin{equation}\label{eq1:2}
g_M := (b_k-a_k)_{1\le k\le n} \in \Z^n.
\end{equation}
We show (Proposition~\ref{prop:g-vector}) that the $g$-vector of $x(\beta)$ 
coincides with $g_{M(\beta)}$.
To summarize, we have:
\begin{Thm}
For $\beta\in\Delta^+(C)$, the $F$-polynomial and $g$-vector of the cluster 
variable $x(\beta)$ with respect to the initial seed 
$(u_1,\ldots,u_n)$ are $F_{M(\beta)}$ and $g_{M(\beta)}$, as defined by 
equations (\ref{eq1:1}) and (\ref{eq1:2}).
\end{Thm}

We conclude this introduction by mentioning related work in the literature.
Several methods have already been explored to obtain generalizations of the 
Caldero-Chapoton formula to the skew-symmetrizable case.
Demonet \cite{D} has used $\Gamma$-equivariant categories of modules over 
preprojective algebras (where $\Gamma$ is the cyclic group generated by a 
diagram automorphism) to obtain cluster characters for acyclic symmetrizable 
cluster algebras in the spirit of~\cite{GLSAdv}, see~\cite[Theorem C]{D}.
In \cite{R1,R2}, Rupel has used categories of representations of valued 
quivers over finite fields
to obtain a quantum analogue of the Caldero-Chapoton formula for acyclic 
symmetrizable cluster algebras.
Theorem~\ref{Thm1} provides yet another approach based on categories of 
locally free $H$-modules.

The paper is organized as follows. Section~\ref{sec:recallpreproj} recalls the 
definition of the algebras $H$ and of the convolution algebras $\cM(H)$, 
as well as the main results from~\cite{GLS1,GLS3} which we will need. 
In Section~\ref{sect:multiplicative} we prove Theorem~\ref{Thm1}~(b).
Sections~\ref{sec:filtrations} and \ref{sec:PBW} contain the proof of 
Theorem~\ref{Thm2}.
In Section~\ref{sect-Fpoly-gvect} we introduce the $g$-vector $g_M$ and the 
$F$-polynomial $F_M$ of a locally free
$H$-module $M$, and we express $X_M$ in terms of $g_M$ and $F_M$. 
Section~\ref{sec:gvectors} shows 
that $g_{M(\beta)}$ is equal to the $g$-vector of the cluster variable $x(\beta)$.
Sections~\ref{sec:otherexprFM}, \ref{sec:varphi-minor}, and~\ref{sec:compare} 
contain the proof that $F_{M(\beta)}$ is equal to the 
$F$-polynomial of $x(\beta)$. 
Section~\ref{sec:thm1d} concludes with the proof of Theorem~\ref{Thm1}~(d).
Finally Section~\ref{sec:examples} illustrates Theorem~\ref{Thm1} with examples.


\section{Quivers with relations associated with Cartan matrices}\label{sec:recallpreproj}


In this section, we recall some definitions and results from \cite{GLS1}.

\subsection{The algebras $H(C,D,\Omega)$}\label{subsec:DefH}
Let $C = (c_{ij}) \in M_n(\Z)$ be a Cartan matrix. This means that
\begin{itemize}

\item[(C1)]
$c_{ii} = 2$ for all $i$;

\item[(C2)] 
$c_{ij} \le 0$ for all $i \not= j$;

\item[(C3)]
There is a diagonal integer matrix $D = \diag(c_1,\ldots,c_n)$ with
$c_i \ge 1$ for all $i$ such that
$DC$ is a symmetric \emph{positive definite} matrix. 

\end{itemize}
The matrix $D$ appearing in (C3) is called a \emph{symmetrizer} of $C$.
The symmetrizer $D$ is \emph{minimal} if $c_1 + \cdots + c_n$ is
minimal. 

An \emph{orientation of} $C$ is a subset 
$\Omega \subset \{1,\ldots,n\} \times \{1,\ldots,n\}$
such that  
the following hold:

\begin{itemize}

\item[(i)]
$\{ (i,j),(j,i) \} \cap \Omega \not= \varnothing$
if and only if $c_{ij}<0$;

\item[(ii)]
If $(i,j)\in\Omega$ then $(j,i)\not\in\Omega$.


\end{itemize}
For an orientation $\Omega$ of $C$ let
$Q := Q(C,\Omega) := (Q_0,Q_1,s,t)$ be the quiver with the
set of vertices $Q_0 := \{ 1,\ldots, n\}$ and 
with the set of arrows 
\[
Q_1 := \{ \alpha_{ij}\df j \to i \mid (i,j) \in \Omega \}
\cup \{ \vep_i\df i \to i \mid 1 \le i \le n \}.
\]
(Thus we have $s(\alpha_{ij}) = j$ and $t(\alpha_{ij}) = i$ and
$s(\vep_i) = t(\vep_i) = i$, where $s(a)$ and $t(a)$ denote the
starting and terminal vertex of an arrow $a$, respectively.)
Let $Q^\circ := Q^\circ(C,\Omega)$ be the quiver obtained from $Q$ by deleting
all loops $\vep_i$.
Clearly,
$Q^\circ$ is an acyclic quiver.

For an orientation $\Omega$ of $C$ and some $1 \le i \le n$ let
\[
s_i(\Omega) :=
\{ (r,s) \in \Omega \mid i \notin \{r,s\} \}
\cup \{ (s,r) \in \Omega^* \mid i \in \{r,s\}  \},
\]
where $\Omega^* := \{(s,r) \mid (r,s) \in \Omega\}$ denotes the orientation 
opposite to $\Omega$.

For a quiver $Q = Q(C,\Omega)$ and a symmetrizer $D = \diag(c_1,\ldots,c_n)$ of 
$C$, let
\[
H := H(C,D,\Omega) := KQ/I
\] 
where $KQ$ is the path algebra of $Q$ over a field $K$, and $I$ is the ideal 
of $KQ$ defined by the following relations:
\begin{itemize}

\item[(H1)]
For each $i$ we have
\[
\vep_i^{c_i} = 0;
\]

\item[(H2)]
For each $(i,j) \in \Omega$ we have
\[
\vep_i^{|c_{ji}|}\alpha_{ij} = \alpha_{ij}\vep_j^{|c_{ij}|}.
\]

\end{itemize}

Let $\rep(H)$ be the category of finite-dimensional $H$-modules.
Such a module $M$ is given by a $K$-vector space $M_i$ at each vertex $i$ of 
$Q$, and linear maps $M(a) : M_{s(a)}\to M_{t(a)}$ for each arrow $a$ of $Q$ 
satisfying the defining relations of $H$. In particular the endomorphism 
$M(\varepsilon_i)$ endows $M_i$ with the structure of an $H_i$-module, where
\[
H_i := K[\vep_i]/(\vep_i^{c_i})
\] 
is a truncated polynomial ring.
An $H$-module $M$ is \emph{locally free} if $M_i$ is a free
$H_i$-module for all $i$.
The rank of a free $H_i$-module $M_i$ is denoted by $\rk(M_i)$.
For a locally free $H$-module $M$ let $\rkv(M) := (\rk(M_1),\ldots,\rk(M_n))$
be the \emph{rank vector} of $M$.
Let $\rep_\vp(H) \subseteq \rep(H)$ be the subcategory of
locally free $H$-modules. 
In particular, let $E_i := H_i$ seen as an $H$-module.

We refer to \cite{GLS1} for further details.

\subsection{Recollection of results}\label{subsec:recoll}

\begin{Thm}[{\cite[Theorem~1.2]{GLS1}}]\label{thma:GLS1}
The algebra $H$ is a $1$-Iwanaga-Gorenstein algebra.
For $M \in \rep(H)$ the following are equivalent:
\begin{itemize}

\item[(i)]
$\pdim(M) \le 1$;

\item[(ii)]
$\idim(M) \le 1$;

\item[(iii)]
$\pdim(M) < \infty$;

\item[(iv)]
$\idim(M)  < \infty$;

\item[(v)]
$M$ is locally free.

\end{itemize}
\end{Thm}
Define a bilinear form 
\[
\bil{-,-}_H\df \Z^n \times \Z^n \to \Z
\] 
by 
\[
\bil{\alpha_i,\alpha_j}_H := 
 \left\{
 \begin{matrix}
c_ic_{ij} & \mbox{if $(j,i)\in\Omega$,}\\
c_{i} & \mbox{if $i=j$,}\\
0 & \mbox{otherwise.}
 \end{matrix}
\right.
\]
(Here $\alpha_1,\cdots,\alpha_n$ denotes the standard basis of $\Z^n$.)
For $M,N \in \rep_\vp(H)$ we have
\[
\bil{\rkv(M),\rkv(N)}_H = \dim \Hom_H(M,N) - \dim \Ext_H^1(M,N),
\]
see~\cite[Section~4]{GLS1}. 
An $H$-module $M$ is said to be \emph{rigid} if $\Ext_H^1(M,M) = 0$.

Let $\tau_H$ be the Auslander-Reiten
translation for the algebra $H$, and let $\tau_H^-$ be the  
inverse Auslander-Reiten translation.
An indecomposable $H$-module $M$ is \emph{preprojective}
(resp. \emph{preinjective}) if there exists some
$k \ge 0$ such that $M \cong \tau_H^{-k}(P)$ (resp.
$M \cong \tau_H^k(I)$) for some indecomposable projective
$H$-module $P$ (resp. indecomposable injective
$H$-module $I$). 
(The usual definition of a preprojective or preinjective module $M$
requires some additional conditions on the Auslander-Reiten
component containing $M$.)
An indecomposable $H$-module $M$ is called $\tau$-\emph{locally free}, if
$\tau_H^k(M)$ is locally free for all $k \in \Z$.

Since we assume that $C$ is a Cartan matrix, we can state: 

\begin{Thm}[{\cite[Theorem~1.3]{GLS1}}]\label{thmb:GLS1b}
There are only finitely many isomorphism classes of $\tau$-locally free 
$H$-modules, and the following hold:
\begin{itemize}

\item[(i)]
The map 
$M \mapsto \rkv(M)$ 
yields a bijection between the set of isomorphism classes of 
$\tau$-locally free $H$-modules
and the set $\Delta^+(C)$ of positive roots of the semisimple 
complex Lie algebra associated with $C$.

\item[(ii)]
For an indecomposable $H$-module $M$ the following are equivalent:
\begin{itemize}

\item[(a)]
$M$ is preprojective;

\item[(b)]
$M$ is preinjective;

\item[(c)]
$M$ is $\tau$-locally free;

\item[(d)]
$M$ is locally free and rigid.

\end{itemize}
\end{itemize}
\end{Thm}
For $\beta \in \Delta^+(C)$, let $M(\beta)$ be the
indecomposable locally free rigid $H$-module with $\rkv(M(\beta)) = \beta$.
Note that for $1 \le i \le n$ we have $M(\alpha_i) = E_i$.

\subsection{Representations of modulated graphs} \label{subsec:species}

With the datum of $(C,D,\Omega)$ and of a field $F$ having field extensions 
$F_i$ of degree $c_i$ for every $1\le i\le n$, 
we can associate as in~\cite[Section 2]{GLS3} a \emph{modulated graph} in the 
sense of Dlab and Ringel \cite{DR1}.
Let $T$ denote, as in~\cite[Section 2]{GLS3}, the tensor $F$-algebra attached 
to this modulated graph. This is a finite-dimensional hereditary $F$-algebra of
finite representation type given by the Cartan type of $C$.
For each positive root $\beta\in\Delta^+(C)$ we write $X(\beta)$ for the unique 
(up to isomorphism) 
indecomposable $T$-module with dimension vector $\beta$. 

In the sequel we will use the following comparison result proved in \cite{GLS3}.

\begin{Prop}[{\cite[Proposition 5.5]{GLS3}}]\label{GLS3-Prop5.5}
For every $\beta, \gamma \in \Delta^+(C)$, we have
\begin{eqnarray*}
\dim_\CC\Hom_H(M(\beta),M(\gamma)) &=& \dim_F\Hom_T(X(\beta),X(\gamma)),\\[1mm]
\dim_\CC\Ext^1_H(M(\beta),M(\gamma)) &=& \dim_F\Ext^1_T(X(\beta),X(\gamma)).
\end{eqnarray*}
\end{Prop}

\subsection{Convolution algebras}\label{subsec:convolution}
In this section, assume that $K = \C$.
For a dimension vector $\bd=(d_1,\ldots,d_n)\in \N^n$, let $\rep(H,\bd)$ denotes
the affine variety of $H$-modules with
dimension vector $\bd$.
Let $G(\bd) = \prod_i \GL(d_i,\C)$, which acts on $\rep(H,\bd)$
by conjugation.
Let 
\[
\cF(H) := \bigoplus_{\bd \in \N^n} \cF(H)_\bd,
\] 
where $\cF(H)_\bd$ is the $\C$-vector space of constructible
functions $\rep(H,\bd) \to \C$ taking constant values on $G(\bd)$-orbits.
For $\bd$ such that $\rep_\vp(H,\bd) \not= \varnothing$ let
\[
\br := \bd/D := (d_1/c_1,\ldots,d_n/c_n)
\]
be the associated rank vector.
Let $\rep_\vp(H,\br)$ be the subvariety of $\rep(H,\bd)$ consisting of
the locally free $H$-modules with rank vector $\br$.

The space $\cF(H)$ is endowed with an associative convolution product defined by
\[
(f*g)(M) := \sum_{m \in \C} m \cdot\chi(\{U\in \rep(H,\bd)\mid U \subseteq M \mbox{ and } f(U)g(M/U) = m \}),
\]
where $f\in \cF(H)_\bd$, $g\in \cF(H)_\be$, $M \in \rep(H,\bd+\be)$, and
$\chi$ denotes the topological Euler characteristic. (Note that the sum is 
finite because $f$ and $g$ are constructible.) 

For $M \in \rep(H)$ define $1_M \in \cF(H)$ by
\[
1_M(N) :=
\begin{cases}
1 & \text{if $M \cong N$},
\\
0 & \text{otherwise}.
\end{cases}
\]
For $1 \le i \le n$ let $\theta_i := 1_{E_i}$.
Let 
\[
\cM(H) = \bigoplus_{\bd \in \N^n} \cM(H)_\bd
\]
be the subalgebra of $\cF(H)$ generated by $\{ \theta_1,\ldots,\theta_n \}$,
where
\[
\cM(H)_\bd := \cF(H)_\bd \cap \cM(H).
\]
For $f \in \cM(H)_\bd$ let 
\[
\supp(f) := \{ M \in \rep(H,\bd) \mid f(M) \not= 0 \}
\]
be the \emph{support} of $f$.
We have $\supp(f) \subseteq \rep_\vp(H,\bd/D)$, \cite[Lemma 4.2]{GLS3}.

By \cite[Theorem~1.1]{GLS3}, there is a Hopf algebra
isomorphism $U(\n) \to \cM(H)$, where
$U(\n)$ is the enveloping algebra of the positive part $\n$ of the 
semisimple Lie algebra $\g$ with Cartan matrix $C$.
This map sends the Chevalley generator $e_i$ of $U(\n)$ to the constructible 
function $\theta_i$.

Recall that the comultiplication in $\cM(H)$ is given by
$\theta_i \mapsto \theta_i \otimes 1 + 1 \otimes \theta_i$.
A constructible function $f \in \cM(H)_\bd$ is primitive
if and only if $\supp(f)$ consists just of indecomposable 
modules,~\cite[Lemma 4.6]{GLS3}.

The algebra $\cM(H)$ is isomorphic to the enveloping algebra
$U(\cP(\cM))$ of the Lie algebra $\cP(\cM) \subset \cM(H)$ of
primitive elements.
We have a root space decomposition
\[
\cP(\cM) = \bigoplus_{\beta \in \Delta^+(C)} \cP(\cM)_\beta
\]
which (under the isomorphism $U(\n) \to \cM(H)$) corresponds
to the root space decomposition
\[
\n = \bigoplus_{\beta \in \Delta^+(C)} \n_\beta.
\]
Since we are in the Dynkin case, we have 
\[
\dim \cP(\cM)_\beta = \dim \n_\beta = 1
\]
for all $\beta \in \Delta^+(C)$.
Using this and \cite[Theorem~6.1]{GLS3} we have

\begin{Thm}[{}]\label{thmc:GLS3}
For each $\beta \in \Delta^+(C)$ there exists a unique primitive
element $\theta_\beta \in \cP(\cM)_\beta$ such that
$\theta_\beta(M(\beta)) = 1$.
\end{Thm}

For a finite-dimensional $\C$-vector space $V$, let $V^*$ be its dual.
Let 
\[
\cM(H)_{\rm gr}^* := \bigoplus_{\bd \in \N^n} \cM(H)_\bd^*
\]
be the \emph{graded dual} of the Hopf algebra $\cM(H)$.
For each locally free $H$-module $M$, let 
$\delta_M \in \cM(H)_{\rm gr}^*$ be the evaluation function defined by
\[
\delta_M(f) := f(M),\qquad (f\in \cM(H)).
\]
It follows from the description of the comultiplication of $\cM(H)$ 
in~\cite[Section 4.2]{GLS3}
that these delta functions are multiplicative, in the sense that for any 
locally free $H$-modules $M$ and $N$ one has
\begin{equation}\label{eq-2.1}
\delta_M \cdot \delta_N = \delta_{M\oplus N}.
\end{equation}


\section{Multiplicativity of $X_M$} \label{sect:multiplicative}

In this section, we give a direct proof that the Laurent polynomials $X_M$ have 
a similar multiplicative property, that is,
of Theorem~\ref{Thm1}~(b).
This could also be deduced from (\ref{eq-2.1}), in view of Lemma~\ref{Lem-5.1} 
and Proposition~\ref{Prop4-2}.
\begin{Lem}\label{lem3}
 Let $M$ and $N$ be locally free $H$-modules. For any rank vector 
$\br\in\N^n$ we have
 \[
  \chi(\Gr_\lf(\br,M\oplus N)) = \sum_{\bs+\bt=\br} \chi(\Gr_\lf(\bs,M))\chi(\Gr_\lf(\bt, N)).
 \]
\end{Lem}

\begin{proof}
 Let $d=\sum_ic_ir_i$ be the corresponding dimension, and consider the ordinary 
Grassmannian parametrizing  all dimension $d$ subspaces of  the vector space 
$M\oplus N$. For $\la\in\C^*$, consider the linear 
automorphism $\phi_\la$ of $M\oplus N$ defined by
 \[
  \phi_\la(x+y) = \la x + y, \qquad (x\in M,\ y\in N).
 \]
 This yields a $\C^*$-action on the ordinary Grassmannian defined by
 \[
  \la \cdot L := \phi_\la(L), 
 \]
whose fixed points are the subspaces $L$ satisfying 
$L = (L\cap M) \oplus (L\cap N)$.

Since $\phi_\la$ is in fact an $H$-module automorphism, if $L$ is a locally free
submodule of $M\oplus N$ then
$\phi_\la(L)$ is also locally free, so this action restricts to a $\C^*$-action 
on $\Gr_\lf(\br,M\oplus N)$.
Suppose that $L$ is a locally free fixed point for this action. Then 
$L = (L\cap M) \oplus (L\cap N)$.
In particular, $L_i = (L\cap M)_i \oplus (L\cap N)_i$ for every $i$. This is a 
free $H_i$-module, so by Krull-Schmidt
$(L \cap M)_i$ and $(L\cap N)_i$ are also free over $H_i$. Hence $L\cap M$ and 
$L\cap N$ are both locally free
$H$-submodules of $M$ and $N$. Now the claimed identity follows from the fact 
that $\chi(\Gr_\lf(\br,M\oplus N))$
is equal to the Euler characteristic of its fixed point subset under the 
$\C^*$-action.
Indeed, by the above discussion this fixed point subset is isomorphic to
\[
 \bigsqcup_{\bs+\bt=\br} \Gr_\lf(\bs,M) \times \Gr_\lf(\bt, N).
\]
\end{proof}

\begin{Cor}\label{cor4}
For locally free $H$-modules
$M$ and $N$ we have
\[
 X_M \cdot X_N = X_{M\oplus N}.
\]
\end{Cor}

\begin{proof}
From the definition of $X_M$ and the bilinearity of $\<\cdot,\cdot\>$ it 
follows that the monomial
\[
 \prod_{i=1}^n v_i^{-\<\br,\,\rkv(E_i)\>-\<\rkv(E_i),\,\rkv(M\oplus N)-\br\>}
\]
occurs in $X_M\cdot X_N$ with coefficient 
\[
\sum_{\bs+\bt=\br} \chi(\Gr_\lf(\bs,M))\chi(\Gr_\lf(\bt, N)). 
\]
By Lemma~\ref{lem3}, this is $\chi(\Gr_\lf(\br,M\oplus N))$, hence the result.
\end{proof}


\section{Dependence on the symmetrizer}
\label{symmetrizer}

The algebra $H=H(C,D,\Omega)$ and its category of locally free modules depend 
a lot on the choice of the symmetrizer $D$.
But as shown by Theorem~\ref{thmb:GLS1b}, the classification of the 
indecomposable rigid locally free modules does
not depend on $D$. 
Moreover, we also have the following result.

\begin{Prop}\label{prop-sym}
For every $\beta \in \Delta^+(C)$, the Laurent polynomial $X_{M(\beta)}$ is 
independent of $D$. 
\end{Prop}
\begin{proof}
By \cite[Corollary 1.3]{GLS2}, the Euler characteristic 
$\chi(\Gr_\lf(\br,M(\beta)))$ is independent of $D$.
The result then follows immediately from the definition of $X_{M(\beta)}$.
\end{proof}

\begin{Cor}\label{cor-sym}
If $C$ is symmetric, then Theorem~\ref{Thm1}~(c) holds.
\end{Cor}
\begin{proof}
If $C$ is symmetric and $D = I_n$, Theorem~\ref{Thm1}~(c) holds by \cite{CC}. 
Hence, by Proposition~\ref{prop-sym}, it also holds for an arbitrary 
symmetrizer $D = kI_n\ (k \ge 1)$. 
\end{proof}


\section{Filtrations of preprojective modules}
\label{sec:filtrations}


The results in this section can be found in Omlor's Master Thesis
\cite{O}.
Using the theory developed in \cite{GLS1},
these are straightforward generalizations of similar results by 
Dlab and Ringel \cite{DR}, who study the representation theory of 
species of Dynkin type.
For convenience, we include proofs.

As before, let $H = H(C,D,\Omega)$ with $C$ of Dynkin type.
We label the positive roots
$\beta_1,\ldots,\beta_r$ in $\Delta^+(C)$ such that
$\Hom_H(M(\beta_i),M(\beta_j)) = 0$ for all $i < j$.
(This is always possible by \cite[Corollary~5.8]{GLS3}.)

\begin{Lem}\label{lem:filt1}
We have
$\Ext_H^1(M(\beta_j),M(\beta_i)) = 0$ for all $j \ge i$.
\end{Lem}

\begin{proof}
We use the Auslander-Reiten formula
\[
\Ext_H^1(M(\beta_j),M(\beta_i)) \cong D\Hom_H(M(\beta_i),\tau_H(M(\beta_j))).
\]
(Here we used that the modules $M(\beta_j)$ have projective dimension
at most one.) 
We have $\tau_H(M(\beta_j)) \cong M(\beta_s)$ for some $s > j$. 
If $\Hom_A(M(\beta_i),\tau_H(M(\beta_j))) \not= 0$ for some $j \ge i$, then 
$i \ge s > j$, a contradiction.
\end{proof}

\begin{Lem}\label{lem:filt2}
Let $M \in \rep(H)$ such that there is a filtration
\[
0 = M_0 \subset M_1 \subset \cdots \subset M_t = M
\]
such that for each $1 \le i \le t$ we have
$M_i/M_{i-1} \cong M(\beta_j)$ for some $1 \le j \le r$.
Let $a_j$ be the number of indices $i$ such that 
$M_i/M_{i-1} \cong M(\beta_j)$.
Then there is a filtration
\[
0 = N_r \subseteq N_{r-1} \subseteq \cdots \subseteq N_1 \subseteq N_0 = M
\]
such that for each $1 \le j \le r$ we have 
$N_{j-1}/N_j \cong M(\beta_j)^{a_j}$.
\end{Lem}

\begin{proof}
This follows by induction and by our observation that 
$\Ext_H^1(M(\beta_j),M(\beta_i)) = 0$ for all $j \ge i$.
\end{proof}

\begin{Lem}\label{lem:filt3}
For $\gamma \in \Delta^+(C)$, let
\[
0 = N_r \subseteq N_{r-1} \subseteq\cdots\subseteq N_1 \subseteq N_0 = M(\gamma)
\]
be a filtration such that for each $1 \le j \le r$ we have 
$N_{j-1}/N_j \cong M(\beta_j)^{a_j}$ for some $a_j \ge 0$.
Suppose that at least two multiplicities $a_j$ are non-zero.
Then there is no filtration of the form
\[
0 = M_0 \subseteq M_1 \subseteq\cdots\subseteq M_{r-1} \subseteq M_r = M(\gamma)
\]
such that for each $1 \le j \le r$ the module 
$M_j/M_{j-1}$ is locally free of rank $a_j\beta_j$.
\end{Lem}

\begin{proof}
Let $s$ be maximal with $a_s \not= 0$.
Thus the first filtration yields a non-zero homomorphism
$M(\beta_s) \to M(\gamma)$.
Suppose for a contradiction that the filtration
\[
0 = M_0 \subseteq M_1 \subseteq\cdots\subseteq M_{r-1} \subseteq M_r = M(\gamma)
\]
exits. Then, we have an epimorphism
$f\df M(\gamma) \to X$ with $X := M_s/M_{s-1}$.
Since we are in the Dynkin case, we have
\[
\bil{\beta_s,\beta_s}_H > 0.
\]
This implies that
\begin{align*}
\bil{\rk(X),\rk(M(\beta_s))}_H &= a_s\bil{\beta_s,\beta_s}_H 
\\
&= \dim \Hom_H(X,M(\beta_s)) - \dim \Ext_H^1(X,M(\beta_s)) > 0.
\end{align*}
Thus 
there is a non-zero homomorphism 
$g\df X \to M(\beta_s)$.
It follows that $g\circ f\df M(\gamma) \to M(\beta_s)$ is
non-zero,
a contradiction to \cite[Corollary~5.8]{GLS3}. 
\end{proof}

Note that in the situation of Lemma~\ref{lem:filt3}, if exactly one $a_j$
is non-zero, then it has to be equal to one, since a non-trivial multiple of a 
root is never a root.

\begin{Lem}\label{lem:filt4}
Let $k$ be a source in $Q(C,\Omega)^\circ$, and let
$\alpha + \alpha_k = \gamma$ with
$\alpha,\gamma \in \Delta^+(C)$.
Then there exists a short exact sequence
\[
0 \to M(\alpha) \to M(\gamma) \to M(\alpha_k) \to 0.
\]
\end{Lem}

\begin{proof}
For any rank vector $\br$, $Z_\br := \rep_\vp(H,\br)$ is a
maximal irreducible component of the generalized nilpotent
variety $\Pi(\br)$ introduced and studied in \cite[Section~4]{GLS4}.

The modules $M(\gamma)$ and $M(\alpha)$ are 
both rigid, and their orbit closures are $Z_\gamma$ and $Z_\alpha$, respectively.

We now use the notation from \cite[Section~5]{GLS4}.
Since $k$ is a source in $Q(C,\Omega)^\circ$ we have
$Z_\gamma \in \Irr(\Pi(\br)^{k,p})^\mx$ with $p = \gamma_k > 0$.
We get
\[
\tilde{f}_k^*(Z_\gamma) = Z_\alpha.
\]
(Like $Z_\gamma$, the irreducible component $\tilde{f}_k^*(Z_\gamma)$
just consists of $H$-modules.
Thus it has to be equal to $Z_\alpha$.)
It follows that there is a short exact sequence
\[
0 \to M(\alpha) \to M(\gamma) \to M(\alpha_k) \to 0.
\]
This finishes the proof.
\end{proof}

There is an obvious dual of Lemma~\ref{lem:filt4} where one assumes
that $k$ is a sink in $Q(C,\Omega)^\circ$.

\begin{Lem}\label{lem:filt5}
Let
$\alpha + \beta = \gamma$ with
$\alpha,\beta,\gamma \in \Delta^+(C)$.
Then there exists a short exact sequence
\[
0 \to M(\alpha) \to M(\gamma) \to M(\beta) \to 0
\]
or a short exact sequence
\[
0 \to M(\beta) \to M(\gamma) \to M(\alpha) \to 0.
\]
\end{Lem}

\begin{proof}
In \cite[Section~9]{GLS1} we developed reflection functors for
generalized preprojective algebras.
These restrict to reflection functors
\[
F_k^+\df \rep(H(C,D,\Omega)) \to \rep(H(C,D,s_k(\Omega))
\]
for $k$ a sink in $Q(C,\Omega)^\circ$, and 
\[
F_k^-\df \rep(H(C,D,\Omega)) \to \rep(H(C,D,s_k(\Omega))
\]
for $k$ a source in $Q(C,\Omega)^\circ$.

Now we can proceed as in \cite{DR}.
Namely, there is a sequence $(i_1,\ldots,i_t)$ such that
$M(\alpha') := F_{i_1}^- \cdots F_{i_t}^-(M(\alpha)) \not= 0$ and
$M(\beta') := F_{i_1}^- \cdots F_{i_t}^-(M(\beta)) \not= 0$, and 
one if these modules is locally free injective of rank $\alpha_k$
for some $k$.
Thus,
without loss of generality we can assume that $M(\beta') \cong E_k$ with
$E_k$ injective. 
Set $M(\gamma') := F_{i_1}^- \cdots F_{i_t}^-(M(\gamma))$. 
(Note that the modules $M(\alpha')$, $M(\beta')$ and $M(\gamma')$ are modules
over $H(C,D,s_{i_1} \cdots s_{i_t}(\Omega))$.)
Applying Lemma~\ref{lem:filt4}
we get a short exact sequence
\[
0 \to M(\alpha') \to M(\gamma') \to E_k \to 0.
\]
Since the functors $F_i^+$ are exact on short exact sequences
of $\tau$-locally free modules which do not have non-zero projective direct 
summands isomorphic to $E_i$, we can apply
$F_{i_t}^+ \cdots F_{i_1}^+$ to the sequence above and get
a short exact sequence
\[
0 \to M(\alpha) \to M(\gamma) \to M(\beta) \to 0.
\]
Note that in the entire proof we strongly used 
\cite[Proposition~9.6]{GLS1} and 
\cite[Theorem~11.10]{GLS1}, 
compare also the discussion in \cite[Section~5]{GLS3}.
\end{proof}

A proof of the following lemma can be found in 
\cite[Ch. VI, \S 1, Proposition~19]{B}.

\begin{Lem}\label{lem:filt6}
Let $(\gamma,\gamma_1,\ldots,\gamma_t)$ be a sequence of
positive roots in $\Delta^+(C)$ such that
\[
\gamma = \gamma_1 + \cdots + \gamma_t.
\]
Then there is a permutation $\pi$ of $\{ 1,\ldots,t \}$ 
such that 
\[
\sum_{i=1}^s \gamma_{\pi(i)}
\]
is a positive root for each $1 \le s \le t$.
\end{Lem}

\begin{Thm}\label{thm:filt7}
Let $(\gamma,\gamma_1,\ldots,\gamma_t)$ be a sequence of
positive roots in $\Delta^+(C)$ such that
\[
\gamma = \gamma_1 + \cdots + \gamma_t.
\]
Then there is a permutation $\pi$ of $\{ 1,\ldots,t \}$ and a 
filtration
\[
0 = M_0 \subset M_1 \subset \cdots \subset M_t = M(\gamma)
\]
such that $M_i/M_{i-1} \cong M(\gamma_{\pi(i)})$ for all $1 \le i \le t$.
\end{Thm}

\begin{proof}
The proof is a straightforward induction using 
Lemmas~\ref{lem:filt5} and \ref{lem:filt6}. 
\end{proof}


\section{PBW bases for convolution algebras}
\label{sec:PBW}


We keep our numbering convention for the positive roots
$\beta_1,\ldots,\beta_r$.
Recall from Theorem~\ref{thmc:GLS3} the constructible functions 
$\theta_{\beta_i}$, and for $a_i \ge 0$ define
\[
\theta_{\beta_i}^{(a_i)} := \frac{1}{a_i!} (\theta_{\beta_i})^{*a_i}.
\]
For $\ba = (a_1,\ldots,a_r) \in \N^r$ let
\[
\theta_\ba := \theta_{\beta_1}^{(a_1)} * \cdots * \theta_{\beta_r}^{(a_r)}
\text{\;\;\; and \;\;\;}
M_\ba := M(\beta_1)^{a_1} \oplus \cdots \oplus M(\beta_r)^{a_r}.
\]
Then 
\[
\cP := \{ \theta_\ba \mid \ba \in \N^r \}
\]
is a PBW basis of $\cM(H)$.
Set $\delta_\ba := \delta_{M_\ba}$. We can now prove the next theorem, which is 
a reformulation of Theorem~\ref{Thm2}.

\begin{Thm}\label{thm:PBW1}
For $\ba,\bb \in \N^r$ we have
\[
\delta_\ba(\theta_\bb) = \delta_{\ba,\bb}.
\]
\end{Thm}

\begin{proof}
Note that for degree reasons, we have $\delta_\ba(\theta_\bb) = 0$ unless 
$\sum_i a_i\beta_i = \sum_i b_i\beta_i$.
Let us consider first the case when $\ba = e_k$ is the $k$th standard basis 
vector of $\Z^r$.
In other words, we have $M_\ba = M(\beta_k)$.
By Theorem~\ref{thmc:GLS3},
\[
\delta_{M(\beta_k)}(\theta_{\beta_k}) = \theta_{\beta_k}(M(\beta_k)) = 1.
\]
Next, let $\bb = (b_1,\ldots,b_r) \in \N^r$ with $\bb \not= e_k$.
Then 
\[
\delta_{M(\beta_k)}(\theta_\bb) = 
(\theta_{\beta_1}^{(b_1)} * \cdots * \theta_{\beta_r}^{(b_r)})(M(\beta_k)).
\]
We can assume that 
\[
\beta_k = b_1\beta_1 + \cdots + b_r\beta_r.
\]
Now Theorem~\ref{thm:filt7} combined with Lemma~\ref{lem:filt2}
yields a filtration
\[
0 = N_r \subseteq N_{r-1} \subseteq\cdots\subseteq N_1 \subseteq N_0 = M(\beta_k)
\]
such that for each $1 \le j \le r$ we have 
$N_{j-1}/N_j \cong M(\beta_j)^{b_j}$.
Using Lemma~\ref{lem:filt3} we see that
there is no filtration of the form
\[
0 = M_0 \subseteq M_1 \subseteq\cdots\subseteq M_{r-1} \subseteq M_r = M(\beta_k)
\]
such that for each $1 \le j \le r$ the module 
$M_j/M_{j-1}$ is locally free of rank $b_j\beta_j$.
But this implies that
\begin{equation}\label{eq:6.1}
(\theta_{\beta_1}^{(b_1)} * \cdots * \theta_{\beta_r}^{(b_r)})(M(\beta_k)) = 0.
\end{equation}
Thus $\delta_{M(\beta_k)}$ is a dual PBW generator, and the theorem is proved in 
the case $\ba = e_k$.

The general case now follows from general Hopf algebra considerations.
Indeed, by (\ref{eq-2.1}), for $\ba\in\N^r$ we have  
\[
\delta_\ba = (\delta_{M(\beta_1)})^{a_1} \cdots (\delta_{M(\beta_r)})^{a_r}.
\]
Now, if $\varphi_1,\ldots,\varphi_k \in \cM(H)^*_{\rm gr}$ and $f\in \cM(H)$, 
we have
\[
(\varphi_1\cdots\varphi_k)(f) = (\varphi_1\otimes\cdots \otimes\varphi_k)(\Delta_k(f)),
\]
where $\Delta_k: \cM(H) \to \cM(H)^{\otimes k}$ denotes the iterated 
comultiplication. Hence
\[
 \delta_\ba(\theta_\bb) = (\delta_{M(\beta_1)})^{\otimes a_1}\otimes \cdots \otimes (\delta_{M(\beta_r)})^{\otimes a_r}
 \left(\Delta_{a_1+\cdots+a_r}\left( \theta_{\beta_1}^{(b_1)} * \cdots * \theta_{\beta_r}^{(b_r)}\right)\right).
\]
Since $\theta_{\beta_k}$ is a primitive element of $\cM(H)$,
\[
\Delta_{a_1+\cdots+a_r}( \theta_{\beta_k}) = (\theta_{\beta_k}\otimes 1 \otimes \cdots \otimes 1) 
+ (1\otimes\theta_{\beta_k}\otimes 1 \otimes \cdots \otimes 1)
+ \cdots 
+ (1\otimes \cdots \otimes 1 \otimes \theta_{\beta_k}).
\]
Hence, using (\ref{eq:6.1}) and the fact that $\Delta_{a_1+\cdots+a_r}$ 
is an algebra homomorphism,
we see that $\delta_\ba(\theta_\bb)$ can be nonzero only if $\bb = \ba$.
Moreover in that case, since the coefficient of 
\[
 \theta_{\beta_1}^{\otimes a_1}\otimes\cdots \otimes \theta_{\beta_r}^{\otimes a_r}
\]
in
\[
 \Delta_{a_1+\cdots+a_r}\left( \theta_{\beta_1}^{(a_1)} * \cdots * \theta_{\beta_r}^{(a_r)}\right)
 =
 \frac{1}{a_1!\cdots a_r!} \Delta_{a_1+\cdots+a_r}\left( \theta_{\beta_1}^{a_1} * \cdots * \theta_{\beta_r}^{a_r}\right)
\]
is equal to 1, we get $\delta_\ba(\theta_\ba)=1$.
This finishes the proof.
\end{proof}

\begin{Cor}\label{cor:PBW2}
The set
\[
\cP^* = \{ \delta_\ba \mid \ba \in \N^r \} \subset \cM(H)_{\rm gr}^*
\]
is the dual of the PBW basis 
\[
\cP = \{ \theta_\ba \mid \ba \in \N^r \} \subset \cM(H).
\]
\end{Cor}


\section{$F$-polynomials and $g$-vectors}\label{sect-Fpoly-gvect}


Fomin and Zelevinsky \cite{FZ4}  have shown how to express cluster variables 
in terms of $F$-polynomials and $g$-vectors. 
Let us recall this formula for the coefficient-free cluster algebra $\AA$.

For $j = 1,\ldots,n$, define: 
\[
z_j = \prod_{i=1}^{n} u_i^{b_{ij}}.                              
\]
To every cluster variable $x(\beta)$, Fomin and Zelevinsky attach a 
\emph{$g$-vector} 
$g_\beta\in\Z^n$ \cite[Section 6]{FZ4}, and an $F$-polynomial 
$F_\beta \in \Z[t_1,\ldots,t_n]$ \cite[Definition 3.3]{FZ4}.
One then has \cite[Corollary 6.3]{FZ4}:
\begin{equation}\label{eq-g-F}
x(\beta) = u^{g_\beta} F_\beta(z), 
\end{equation}
where $u^{g_\beta} = u_1^{g_{\beta,1}}\cdots u_n^{g_{\beta,n}}$ and 
$z = (z_1,\ldots,z_n)$.
Note that the denominator of \cite[Corollary 6.3]{FZ4} does not occur 
in~\eqref{eq-g-F} because $\AA$ is coefficient-free. 

\begin{Lem}\label{Lem-5.1}
Let $M$ be a locally free $H$-module with rank vector $\rkv(M) = (m_i)$.
Define $g_M \in \Z^n$ by:
\begin{equation}\label{eq-gM}
 g_{M,i} = -m_i + \sum_{j\in\Omega(-,i)} m_j |c_{ij}|,\qquad (1\le i \le n),
\end{equation}
and $F_M\in\Z[t_1,\ldots,t_n]$ by:
\[
 F_M(t_1,\ldots,t_n) = \sum_{\br\in\N^n} \chi(\Gr_\lf(\br,M)) t_1^{r_1}\cdots t_n^{r_n}. 
\]
Then 
\[
 X_M = u^{g_M}F_M(z). 
\]
\end{Lem}

\begin{proof}
We have 
\[
 z_j = \prod_{i=1}^{n} u_i^{b_{ij}} = \prod_{i\in\Omega(j,-)} u_i^{c_{ij}}\prod_{i\in\Omega(-,j)}u_i^{-c_{ij}},
\]
hence, for $\br = (r_i) \in \N^n$, 
\[
 z^\br = \prod_j \prod_{i\in\Omega(j,-)}u_i^{r_jc_{ij}} \prod_{i\in\Omega(-,j)}u_i^{-r_jc_{ij}}
 = \prod_i \prod_{j\in\Omega(-,i)}u_i^{-r_j|c_{ij}|} \prod_{j\in\Omega(i,-)}u_i^{r_j|c_{ij}|}.
\]
On the other hand, using the definition of $\<\cdot, \cdot\>$ we have
\begin{equation}
\prod_{i=1}^n v_i^{-\<\br,\,\rkv(E_i)\>-\<\rkv(E_i),\,\rkv(M)-\br\>}
= \prod_i u_i^{-m_i} \prod_{i}\prod_{j\in \Omega(-,i)}\prod_{k\in\Omega(i,-)} u_i^{(m_j-r_j)|c_{ij}|+r_k|c_{ik}|}. 
\end{equation}
Therefore 
using the definition of $g_M$, we get
\[
\prod_{i=1}^n v_i^{-\<\br,\,\rkv(E_i)\>-\<\rkv(E_i),\,\rkv(M)-\br\>} = u^{g_M}z^\br, 
\]
which yields
\[
 X_M = u^{g_M} F_M(z),
\]
as required.
\end{proof}

Therefore, the proof of Theorem~\ref{Thm1}~(c) is reduced to proving that, for 
every positive root $\beta$ of $C$, 
we have
\begin{eqnarray}
 g_{M(\beta)} &=& g_\beta,\\ \label{eq-3.2}
 F_{M(\beta)} &=& F_{\beta}. \label{eq-3.3}
\end{eqnarray}


\section{$g$-vectors and injective coresolutions} \label{sec:gvectors}

Let $M$ be a locally free $H$-module. By Theorem~\ref{thma:GLS1}, $M$ has an 
injective coresolution of length 1:
\[
 0 \to M \to I^0 \to I^1 \to 0.
\]
Let $I^0 = \bigoplus_{1\le k\le n} I_k^{a_k}$ and 
$I^1 = \bigoplus_{1\le k\le n} I_k^{b_k}$ be the decompositions
of $I^0$ and $I^1$ into indecomposable summands.
\begin{Prop}
 We have $g_M = (b_k-a_k)_{1\le k\le n}$.
\end{Prop}
\begin{proof}
We have $\<\rkv(E_i),\rkv(I_j)\>_H = \delta_{ij}c_i$, by definition of the 
homological bilinear form $\<\cdot,\cdot\>_H$. 
Comparing this equation with (\ref{eq-gM}) and taking into account the explicit
expression of $\<\cdot,\cdot\>_H$ 
given in Section~\ref{subsec:recoll}, we get that 
\[
 g_{I_j} = -[\delta_{jk}]_{1\le k\le n}.
\]
For an arbitrary locally free $H$-module, we have 
$\rkv(M) = \rkv(I^0) - \rkv(I^1)$, and the result follows from the
fact that, by (\ref{eq-gM}), the map $\rkv(M) \mapsto g_M$ is $\Z$-linear. 
\end{proof}

From now on, we fix a $(+)$-admissible sequence 
$\bi=(i_1,\ldots,i_n)$~\cite[Section 2.5]{GLS1}.
This means that $\{i_1,\ldots,i_n\} = \{1,\ldots,n\}$, $i_1$ is a sink of the 
acyclic quiver $Q^\circ(C,\Omega)$, 
and $i_k$ is a sink of $Q^\circ(C,s_{i_{k-1}}\cdots s_{i_{1}}(\Omega))$ for 
$2\le k\le n$.
Let $c=s_{i_1}\cdots s_{i_n}$. Then $c$ is a Coxeter element of $W$.
Moreover, $c$ is such that the initial seed of
the cluster algebra $\cA$ which we have defined from the Cartan matrix $C$ and 
the orientation $\Omega$, 
coincides with the initial seed of the cluster algebra of \cite{YZ} defined 
from $C$ and $c$.

\begin{Prop}\label{prop:g-vector}
For $\beta\in\Delta^+(C)$, we have $g_{M(\beta)} = g_\beta$. 
\end{Prop}
\begin{proof}
By \cite[Theorem 1.8]{YZ}, the linear map sending $g_\beta$ to $\beta$ can be 
described in the following way.
Let $\gamma_\beta$ be the weight of $\g$ whose coordinates on the basis 
$\{\varpi_k\}$ of fundamental weights
are the components of $g_\beta$. Then, $c^{-1}(\gamma_\beta)-\gamma_\beta = \beta$.

We know \cite[Lemma 3.3]{GLS1} that the indecomposable injective $H$-module 
$I_{i_k}$ has rank vector
\[
 \rkv(I_{i_k}) = s_{i_n}\cdots s_{i_{k+1}}(\alpha_{i_k}).
\]
Also, by the previous proposition $g_{I_{i_k}}$ expressed as above as a weight 
is equal to $-\varpi_{i_k}$. Now
\[
 c^{-1}(-\varpi_{i_k}) = -s_{i_n}\cdots s_{i_{k}}(\varpi_{i_k})
 = -s_{i_n}\cdots s_{i_{k+1}}(\varpi_{i_k}-\alpha_{i_k}) = -\varpi_{i_k} + \rkv(I_{i_k}).
\]
This proves that $g_{I_k} = g_{\rkv(I_k)}$ for $1\le k\le n$. The result follows 
since $(\rkv(I_k))_{1\le k\le n}$
is a basis of $\Z^n$.
\end{proof}


\section{Another expression of $F_M$} \label{sec:otherexprFM}

From now on, we fix a $(+)$-admissible sequence $\bi=(i_1,\ldots,i_n)$. 

For a locally free $H$-module $M$ and an arbitrary sequence 
$\bj=(j_1,\ldots,j_s)$ of elements of $\{1,\ldots,n\}$,   
let $\cF_{\bj}(M)$ denote the quasi-projective variety of filtrations
\[
 \mathfrak{f} : 0 = M_0 \subset M_1 \subset \cdots \subset M_s = M,
\]
where $M_k/M_{k-1} \cong E_{j_k}$ for $1\le k \le s$.

\begin{Prop}\label{Prop-4.1}
Let $M$ be a locally free $H$-module with rank vector $\bm\in\N^n$.
Let $\br \in \N^n$ with $r_i \le m_i$ for $1\le i\le n$. Set 
$\bj_\br := (i_1^{r_{i_1}},\ldots,i_n^{r_{i_n}},i_1^{m_{i_1}-r_{i_1}},\ldots,i_n^{m_{i_n}-r_{i_n}})$,
where $i_k^a$ means $i_k$ repeated $a$ times.
The map $\pi : \cF_{\bj_\br}(M) \to \Gr_\lf(r,M)$ given by 
$\pi(\mathfrak{f}) = M_{r_1+\cdots+r_n}$
is a surjective morphism whose fibers all have the same Euler characteristic 
$r_1!(m_1-r_1)!\cdots r_n!(m_n-r_n)!$.
Hence 
\[
 \chi(\Gr_\lf(\br,M)) = \frac{1}{r_1!(m_1-r_1)!\cdots r_n!(m_n-r_n)!} \chi(\cF_{\bj_\br}(M)).
\]
\end{Prop}

\begin{proof}
Let us introduce the intermediate variety $\cF'_{\bj_\br}(M)$ of filtrations
\[
 \mathfrak{f}' : 0 = M'_0 \subset M'_1 \subset \cdots \subset M'_{2n} = M,
\]
where $M'_k/M'_{k-1} \cong E_{i_k}^{r_k}$ and $M'_{k+n}/M'_{k+n-1} \cong E_{i_k}^{m_k-r_k}$ for $1\le k \le n$.
We can factor $\pi = \pi' \circ \pi''$, where $\pi' : \cF'_{\bj_\br}(M) \to \Gr_\lf(\br,M)$ is defined by
$\pi'(\mathfrak{f}') = M'_n$ and $\pi'' : \cF_{\bj_\br}(M) \to \cF'_{\bj_\br}(M)$ is defined by
\[
 \pi''(\mathfrak{f}) = M_0 \subset M_{r_{i_1}} \subset M_{r_{i_1}+r_{i_2}} \subset \cdots \subset M.
\]
Let $N$ be a locally free submodule of $M$ with rank $\br$. 
Since $i_1$ is a sink, $N$ has a unique submodule $M'_1$ isomorphic to 
$E_{i_1}^{r_{i_1}}$. 
More generally, the assumption on $\bi$ implies that there is a unique 
filtration 
\[
 0 = M'_0\subset M'_1 \subset \cdots \subset M'_{n-1} \subset M'_n= N  
\]
with $M'_k/M'_{k-1}\cong E_{i_k}^{r_{i_k}}$ for $1\le k\le n$. Similarly, 
there is a unique filtration
\[
 N = M'_n\subset M'_{n+1} \subset \cdots \subset M'_{2n} = M  
\]
with $M'_{n+k}/M'_{n+k-1}\cong E_{i_k}^{m_{i_k}-r_{i_k}}$ for $1\le k\le n$. 
Therefore $\pi'$ is an isomorphism. 

Consider now the fiber $(\pi'')^{-1}(\mathfrak{f}')$ of $\pi''$ over 
$\mathfrak{f}'$. To determine $\mathfrak{f}$ in this fiber, we
must first choose a submodule $M_1$ of $M'_1$ isomorphic to $E_{i_1}$. 
Since $M'_{1} \cong E_{i_1}^{r_{i_1}}$, we claim that the variety of submodules of
$M'_1$ isomorphic to $E_{i_1}$ has a natural fibration over
the projective space $\PP(\C^{r_{i_1}})$ with fibers isomorphic to affine spaces 
(Indeed the kernel $K$ of
the structure map $\varepsilon_{i_1}$ on $M'_1$ is an $r_{i_1}$-dimensional 
subspace, and the kernel of the restriction
of $\varepsilon_{i_1}$ to the submodule $M_1$ is a one-dimensional line $L$ 
inside $K$. The map 
$M_1 \mapsto L$ is the required fibration, with fibers isomorphic to 
$\C^{(c_{i_1}-1)(r_{i_1}-1)}$.)
More generally, the variety of filtrations 
\[
M_1\subset M_2 \subset \cdots\subset M_{r_{i_1}}=M'_1
\]
with $M_k/M_{k-1} \cong E_{i_1}$ has a natural
fibration over the flag variety $\Fl(\C^{r_{i_1}})$, with fibers isomorphic to 
affine spaces.
Repeating this for every step of $\mathfrak{f}'$, we get that 
$(\pi'')^{-1}(\mathfrak{f}')$
has a fibration with basis the product of flag varieties
\[
\Fl(\C^{r_{i_1}})\times \Fl(\C^{r_{i_2}}) \times\cdots\times \Fl(\C^{m_{i_n}-r_{i_n}}), 
\]
and fibers isomorphic to affine spaces.
The proposition now follows from the fact that $\chi(\Fl(\C^{k}))= k!$. 
\end{proof}

Let $G$ be a complex simple and simply connected algebraic group with Cartan 
matrix~$C$.
Let $N$ be a maximal unipotent subgroup of $G$ with Lie algebra $\n$. 
It is well-known that the algebra $\C[N]$ of 
polynomial functions on $N$ can be regarded as the graded dual of the 
Hopf algebra $U(\n)$.
Hence we get a natural isomorphism $\upsilon : \cM(H)_{\rm gr}^* \to \C[N]$. 

For a locally free $H$-module $M$, 
let $\varphi_M = \upsilon(\delta_M)$ denote the polynomial function on $N$ 
corresponding to the linear form $\delta_M$.
It can be described more explicitly as follows. Let $e_i\ (1\le i\le n)$ be 
the Chevalley generators of $U(\n)$, and let 
\[
x_i(t) = \exp_i(te_i), \qquad (t\in \C,\ 1\le i\le n), 
\]
be the corresponding one-parameter subgroups of $N$. Then, for arbitrary 
sequences 
$(j_1,\ldots,j_k) \in \{1,\ldots,n\}^k$ and $(t_1,\ldots,t_k)\in \C^k$ we have
\begin{equation}
 \varphi_M(x_{j_1}(t_1)\cdots x_{j_k}(t_k)) =
 \sum_{(a_1,\ldots,a_k)\in\N^k} \chi\left(\cF_{\left(j_1^{a_1},\ldots,j_k^{a_k}\right)}(M)\right)
 \frac{t_1^{a_1}\cdots t_k^{a_k}}{a_1!\cdots a_k!}.
\end{equation}
(Of course, only finitely many tuples $(a_1,\ldots,a_k)$ can give a nonzero 
contribution to this sum.)
Using Proposition~\ref{Prop-4.1}, we then deduce immediately the next 
proposition.
\begin{Prop}\label{Prop4-2}
Let $\bi=(i_1,\ldots,i_n)$ be a fixed $(+)$-admissible sequence. 
For every locally free $H$-module $M$ we have
\[
 F_M(t_1,\ldots,t_n) = \varphi_M(x_{i_1}(t_1)\cdots x_{i_n}(t_n)x_{i_1}(1)\cdots x_{i_n}(1)).
\]
\end{Prop}
In view of this proposition, in order to prove (\ref{eq-3.3}) we need to 
understand better the 
polynomial functions $\varphi_{M(\beta)}$ for $\beta\in\Delta^+(C)$. 
This will be done in the next section. Before that we can already prove:

\begin{Cor}\label{cor-poly}
For every locally free $H$-module $M$, one can write $X_M$ as a polynomial in 
the functions $X_{M(\beta)}$ with $\beta \in \Delta^+(C)$.
\end{Cor}
\begin{proof}
Using Theorem~\ref{Thm2} and the isomorphism 
$\upsilon\df \cM(H)_{\rm gr}^* \to \C[N]$, we see that 
$\varphi_M$ is a polynomial in the functions $\varphi_{M(\beta)}$. 
Moreover $\C[N]$ is $\N^n$-graded and $\varphi_M$ is homogeneous of 
multi-degree $\rkv(M)$.
Hence by Proposition~\ref{Prop4-2}, $F_M$ is a polynomial in the functions 
$F_{M(\beta)}$.  
Finally, since $g_M$ is a linear function of $\rkv(M)$, it follows from 
Lemma~\ref{Lem-5.1} that $X_M$ is a polynomial in the functions $X_{M(\beta)}$.
\end{proof}

\begin{Cor}
If $C$ is symmetric, Theorem~\ref{Thm1}~(a) holds. 
\end{Cor}
\begin{proof}
If $C$ is symmetric, we already know by Corollary~\ref{cor-sym} that the 
functions $X_{M(\beta)}$ with $\beta \in \Delta^+(C)$
are the cluster variables of $\cA$. Hence Corollary~\ref{cor-poly} shows that 
for every locally free $H$-module $M$,
we have $X_M\in \cA$.
\end{proof}


\section{$\varphi_{M(\beta)}$ as a generalized minor}\label{sec:varphi-minor}

Let $w_0$ be the longest element of the Weyl group $W$ associated with $C$.
Let $\bj = (j_1,\ldots,j_r)$ be a reduced word for $w_0$ such that $j_1$ is a 
source of $Q^\circ(C,\Omega)$,
and $j_k$ is a source of $Q^\circ(s_{j_k}\cdots s_{j_{1}}(\Omega))$ for 
$2\le k \le r$.
We call such a word $\bj$ a $(-)$-adapted reduced word for $w_0$.
It is well known that the positive roots can be enumerated as follows:
\[
 \beta_1 = \alpha_{j_1},\quad \beta_2 = s_{j_1}(\alpha_{j_2}),\quad \ldots, \quad \beta_r=s_{j_1}\cdots s_{j_{r-1}}(\alpha_{j_r}).
\]

\begin{Lem}\label{lem10.1}
Such an ordering satifies the condition : 
\[
\Hom_H(M(\beta_k),M(\beta_l)) = 0\quad \mbox{for all}\quad k < l.
\]
\end{Lem}



\begin{proof}
For $1\le k\le r$, let us write for short
\[
\bil{-,-}_k:=\bil{-,-}_{H(C,D, s_{j_k}\cdots s_{j_1}(\Omega))}.
\]
We also put $\bil{-,-}_0:=\bil{-,-}_{H(C,D,\Omega)}$. 
It is  easy to see that we then have  
\[
\bil{s_{j_k}(\alpha), s_{j_k}(\beta)}_k=\bil{\alpha,\beta}_{k-1},\qquad 
(1\le k\le r,\ \alpha,\beta\in\Z^n). 
\]
Thus, for $k<l$ we get
\[
\bil{\beta_k,\beta_l}_H=\bil{\alpha_{j_k}, s_{j_k}\cdots s_{j_{l-1}}(\alpha_{j_l})}_{k-1}
=\bil{- s_{j_{l-1}}\cdots s_{j_{k+1}}(\alpha_{j_k}),\alpha_{j_l}}_{l-1}\leq 0.
\]
Here, the last inequality holds since $j_l$ is by definition a source for
$s_{j_{l-1}}\cdots s_{j_1}(\Omega)$, and $s_{j_{l-1}}\cdots s_{j_{k+1}}(\alpha_{j_k})$ 
is a positive root. (We agree that 
$s_{j_{l-1}}\cdots s_{j_{k+1}}(\alpha_{j_k}):=\alpha_{j_k}$ in case $l=k+1$.)
A similar argument shows that $\bil{\beta_k,\beta_l}_H\geq 0$ if $k\geq l$.

Note, that we can identify $\bil{-,-}_H$ with the Ringel
bilinear form attached to the tensor $F$-algebra $T$ of 
Section~\ref{subsec:species}. 
Recall that the indecomposable $T$-modules are denoted by 
$X(\beta)\ (\beta\in\Delta^+(C))$. 
Suppose that $\Hom_T(X(\beta_k),X(\beta_l))\neq 0$ for $k<l$. Then, by the 
Auslander-Reiten formula we get
\begin{eqnarray*}
\dim_F\Hom_T(X(\beta_l),\tau_T(X(\beta_k)))&=&\dim_F\Ext^1_T(X(\beta_k),X(\beta_l))\\
&=&\dim_F\Hom_T(X(\beta_k),X(\beta_l))-\bil{\beta_k,\beta_l}_H \\
&\geq& 0.
\end{eqnarray*}
This contradicts the fact that the Auslander-Reiten quiver of the 
representation-finite hereditary
$F$-algebra $T$ has no cycles, see for example~\cite[VIII.5]{ARS}.
Hence we have 
\[
\Hom_T(X(\beta_k),X(\beta_l))= 0 \mbox{  for  } k<l,
\]
and using Proposition~\ref{GLS3-Prop5.5},
we deduce that 
\[
\Hom_H(M(\beta_k),M(\beta_l))= 0 \mbox{  for  } k<l.
\]
\end{proof}
%

\medskip
Hence this total ordering of the roots satisfies the assumption of 
Sections~\ref{sec:filtrations} and \ref{sec:PBW},
so Theorem~\ref{thm:PBW1} shows that the root vectors of the dual PBW-basis 
$\delta_{\ba}$ of $\cM(H)^*_{\rm gr}$ 
defined by this ordering
are the delta functions $\delta_{M(\beta_k)}\ (1\le k\le r)$. 
Therefore, using the isomorphism $\upsilon : \cM(H)_{\rm gr}^* \to \C[N]$, we can
identify the polynomial functions
$\varphi_{M(\beta_k)}\ (1\le k\le r)$ with the root vectors of the dual PBW-basis 
of $\C[N]$ defined by this root ordering.

Recall the generalized minors $\Delta_{u(\varpi_k),v(\varpi_k)} \in \C[G]$ 
introduced by Fomin and Zelevinsky~\cite{FZ}. These are distinguished 
polynomial functions on $G$ defined using Gaussian decomposition.
Here $\varpi_k\ (1\le k\le n)$ denote the fundamental weights of $G$, and 
$u,v$ are elements of the Weyl group $W$.
We will denote by $D_{u(\varpi_k),v(\varpi_k)}$ the restriction of 
$\Delta_{u(\varpi_k),v(\varpi_k)}$ to $N$.

\begin{Prop}\label{Prop-8.2}
For $1\le k\le r$, we have
\[
 \varphi_{M(\beta_k)} = D_{s_{j_1}\cdots s_{j_{k-1}}(\varpi_{i_k}),\ s_{j_1}\cdots s_{j_k}(\varpi_{i_k})}.
\]
\end{Prop}
\begin{proof}
This follows immediately from \cite[Proposition 7.4]{GLS}.
\end{proof}

Combining Proposition~\ref{Prop-8.2} with Proposition~\ref{Prop4-2} we obtain:
\begin{Cor}\label{cor-F-pol}
Let $\bi=(i_1,\ldots,i_n)$ be a $(+)$-admissible sequence, and let 
$\bj=(j_1,\ldots,j_r)$ be a $(-)$-adapted
reduced word for $w_0$. 
For every $1\le k\le r$ we have
\[
 F_{M(\beta_k)}(t_1,\ldots,t_n) = 
 D_{s_{j_1}\cdots s_{j_{k-1}}(\varpi_{j_k}),\ s_{j_1}\cdots s_{j_k}
 (\varpi_{j_k})}(x_{i_1}(t_1)\cdots x_{i_n}(t_n)x_{i_1}(1)\cdots x_{i_n}(1)).
\] 
\end{Cor}


\section{Comparison with \cite{YZ}} \label{sec:compare}

In this section we will deduce (\ref{eq-3.3}) from the results of \cite{YZ}.
By Corollary~\ref{cor-sym}, it is sufficient to consider the case when $C$ is 
\emph{not} symmetric, that is, of type $B_n$, $C_n$, $F_4$ or $G_2$.

Let $\bi=(i_1,\ldots,i_n)$ be a $(+)$-admissible sequence. Then 
$c=s_{i_1}\cdots s_{i_n}$
is a Coxeter element. Let $h$ denote the Coxeter number. 
Since $C$ is not symmetric, $h$ is even and $r = nh/2$. Moreover 
$w_0 = (c^{-1})^{h/2}$~\cite[Chapter 5, \S 6, Cor. 3]{B}.
Hence 
\[
\bj=\underbrace{(i_n,\ldots,i_1,i_n,\ldots,i_1,\ldots,i_n,\ldots,i_1)}_{\text{$r$}}
\]
is a $(-)$-adapted reduced word for $w_0$. From now on, we fix $\bi$ and $\bj$ 
as above.

\begin{Lem}\label{lem:10.1}
For $1\le k\le r$, we write $k = un+v$ with $0\le u < h/2$ and $1\le v\le n$. 
Then
\[
 F_{M(\beta_k)}(t_1,\ldots,t_n) = 
 D_{c^{-u}(\varpi_{i_{n-v+1}}),\ c^{-u-1}(\varpi_{i_{n-v+1}})}
 (x_{i_1}(t_1)\cdots x_{i_n}(t_n)x_{i_1}(1)\cdots x_{i_n}(1)).
\]
\end{Lem}
\begin{proof}
Clearly, $j_k = i_{n-v+1}$. Moreover, for $1\le i,j \le n$ we have 
\[
 s_i(\varpi_j) = \varpi_j - \delta_{ij}\alpha_i.
\]
In particular, if $j\not = i$ then $s_i(\varpi_j) = \varpi_j$. 
Hence $s_{j_1}\cdots s_{j_{k-1}}(\varpi_{j_k}) = c^{-u}(\varpi_{i_{n-v+1}})$ and 
$s_{j_1}\cdots s_{j_{k}}(\varpi_{j_k}) = c^{-u-1}(\varpi_{i_{n-v+1}})$.
The lemma follows then immediately from Corollary~\ref{cor-F-pol}.
\end{proof}

Let us now quote from \cite{YZ} a formula for the $F$-polynomials $F_\beta$ of 
the cluster variables $x(\beta)$ of $\cA$.
Let $\g = \n \oplus \h \oplus \n_-$ be the triangular decomposition of $\g$, 
and let $f_i\ (1\le i\le n)$ 
denote the Chevalley generators of $\n_-$ corresponding to the negative simple 
roots $-\alpha_i$. Let
\[
 y_i(t) = \exp(tf_i),\qquad (1\le i\le n,\ t\in\C)
\]
be the corresponding one-parameter subgroups in $G$.

\begin{Thm}[{\cite[Theorem~1.12]{YZ}}]\label{thm:YZ}
Every $\beta\in\Delta^+(C)$ can be written 
\[
\beta = c^{j-1}(\varpi_i) - c^j(\varpi_i)
\] 
for a unique fundamental weight $\varpi_i$ and $1\le j \le h/2$. Then
\[
 F_\beta(t_1,\ldots,t_n) = \Delta_{c^{j}(\varpi_i),\,c^{j}(\varpi_i)}
 (y_{i_1}(1)\cdots y_{i_n}(1)x_{i_n}(t_n)\cdots x_{i_1}(t_1)).
\]
\end{Thm}
Note that in \cite{YZ}, the principal minors $\Delta_{\varpi_i,\varpi_i}$ 
correspond to the initial cluster
variables $u_1,\ldots,u_n$.
We can now show
\begin{Thm}\label{Thm:11.3}
For $\beta\in\Delta^+(C)$ we have $F_{M(\beta)} = F_\beta$. 
\end{Thm}
 
\begin{proof}
We first slightly rewrite the formula of Theorem~\ref{thm:YZ}. 
Following \cite[Section 3]{YZ}, let $\iota$ denote the involutive 
antiautomorphism of $G$ such that 
\[
 x_i(t)^{\iota} = x_i(t),\qquad y_i(t)^\iota = y_i(t).
\]
It satisfies \cite[(3.7)]{YZ}
\[
 \Delta_{\gamma,\delta}(x) = \Delta_{-\delta,-\gamma}(x^\iota)
\]
for any generalized minor $\Delta_{\gamma,\delta}$, and any $x\in G$.
Also, denoting by $i\mapsto i^*$ the involution on $\{1,\ldots,n\}$ defined
by $\varpi_{i^*} = -w_0(\varpi_i)$, we have
\[
 -c^{j}(\varpi_i) = c^{-h/2+j}(\varpi_{i^*}),\qquad (1\le j \le h/2).
\]
Hence, Theorem~\ref{thm:YZ} can be restated as
\[
F_\beta(t_1,\ldots,t_n) = \Delta_{c^{-h/2+j}(\varpi_{i^*}),\,c^{-h/2+j}(\varpi_{i^*})}
(x_{i_1}(t_1)\cdots x_{i_n}(t_n)y_{i_n}(1)\cdots y_{i_1}(1)),
\]
where $\beta = c^{-h/2+j}(\varpi_{i^*}) - c^{-h/2+j-1}(\varpi_{i^*})$.
In other words, we can write 
\[
\beta = c^{-u}(\varpi_k)-c^{-u-1}(\varpi_k)
\]
for a unique $k\in\{1,\ldots,n\}$ and $0\le u <h/2$, and we have
\begin{equation}\label{eq:10.1}
F_\beta(t_1,\ldots,t_n) = \Delta_{c^{-u}(\varpi_{k}),\,c^{-u}(\varpi_{k})}
(x_{i_1}(t_1)\cdots x_{i_n}(t_n)y_{i_n}(1)\cdots y_{i_1}(1)). 
\end{equation}

On the other hand, by Lemma~\ref{lem:10.1} and \cite[(3.2)]{YZ}, we have 
\begin{equation}\label{eq:10.2}
F_{M(\beta)}(t_1,\ldots,t_n) = \Delta_{c^{-u}(\varpi_{k}),\ c^{-u}(\varpi_{k})}
(x_{i_1}(t_1)\cdots x_{i_n}(t_n)x_{i_1}(1)\cdots x_{i_n}(1)\overline{s_{i_n}}\cdots\overline{s_{i_1}}).
\end{equation}
Recall the following commutation relations in $G$. By \cite[(2.13)]{FZ}, we have
\[
 x_i(1)\overline{s_i} = y_i(1)x_i(-1),\qquad (1\le i\le n).
\]
Also, denoting by $x_\beta(t)$ the one-parameter subgroup of $N$ attached to 
$\beta \in \Delta^+(C)$
(so that $x_i(t) = x_{\alpha_i}(t)$), we have that, if $s_i(\beta) \in \Delta^+(C)$ then
\[
 x_\beta(t)\overline{s_i} = \overline{s_i}x_{s_i(\beta)}(t')
\]
for some $t'\in\C$. Finally, we also have
\[
x_i(t)y_j(t) = y_j(t)x_i(t)\quad\mbox{for $i\not = j$.} 
\]
Using these relations we get
\begin{eqnarray*}
x_{i_1}(1)\cdots x_{i_n}(1)\overline{s_{i_n}}\cdots\overline{s_{i_1}} 
&=&
x_{i_1}(1)\cdots x_{i_{n-1}}(1)y_{i_n}(1)x_{i_n}(-1)\overline{s_{i_{n-1}}}\cdots\overline{s_{i_1}}
\\
&=&
y_{i_n}(1)x_{i_1}(1)\cdots x_{i_{n-1}}(1)\overline{s_{i_{n-1}}}
\cdots\overline{s_{i_1}}x_{s_{i_1}\cdots s_{i_{n-1}}(\alpha_{i_n})}(w_1),
\end{eqnarray*}
for some $w_1\in\C$. Iterating this procedure, we obtain
\[
x_{i_1}(1)\cdots x_{i_n}(1)\overline{s_{i_n}}\cdots\overline{s_{i_1}}
=
y_{i_n}(1)\cdots y_{i_1}(1)
x_{\alpha_{i_1}}(w_n)x_{s_{i_1}(\alpha_{i_2})}(w_{n-1})
\cdots
x_{s_{i_1}\cdots s_{i_{n-1}}(\alpha_{i_n})}(w_1),
\]
for some $w_1,\ldots, w_n \in\C$.
So, writing 
\begin{eqnarray*}
\bx(t)& =& x_{i_1}(t_1)\cdots x_{i_n}(t_n)x_{i_1}(1)\cdots x_{i_n}(1)\overline{s_{i_n}}\cdots\overline{s_{i_1}},
\\
\bz(t) &=& x_{i_1}(t_1)\cdots x_{i_n}(t_n)y_{i_n}(1)\cdots y_{i_1}(1),
\\
\ba &= & x_{\alpha_{i_1}}(w_n)x_{s_{i_1}(\alpha_{i_2})}(w_{n-1})
\cdots
x_{s_{i_1}\cdots s_{i_{n-1}}(\alpha_{i_n})}(w_1),
\end{eqnarray*}
we see that $\bx(t) = \bz(t)\ba$, with $\ba\in N$. Hence it follows from the 
definition of the generalized minors that
\[
 \Delta_{\varpi_k,\varpi_k}(\bx(t)) = \Delta_{\varpi_k,\varpi_k}(\bz(t)),\qquad (1\le k\le n).
\]
Moreover, writing 
$\gamma_k = s_{i_1}\cdots s_{i_{k-1}}(\alpha_{i_k})\ (1\le k\le n)$, we have 
\[
\ba\, \overline{s_{i_n}}\cdots\overline{s_{i_1}} = \overline{s_{i_n}}\cdots\overline{s_{i_1}}\,
x_{c(\gamma_1)}(w'_n)x_{c(\gamma_2)}(w'_{n-1})
\cdots
x_{c(\gamma_n)}(w'_1),
\]
for some $w'_1,\ldots, w'_n \in\C$. Hence, writing 
$\bb = x_{c(\gamma_1)}(w'_n)x_{c(\gamma_2)}(w'_{n-1})\cdots x_{c(\gamma_n)}(w'_1)$, 
we have that 
\begin{eqnarray*}
\overline{s_{i_1}}^{-1}\cdots\overline{s_{i_n}}^{-1}\bx(t) \overline{s_{i_n}}\cdots\overline{s_{i_1}}
&=&
\overline{s_{i_1}}^{-1}\cdots\overline{s_{i_n}}^{-1}\bz(t)\ba \overline{s_{i_n}}\cdots\overline{s_{i_1}}
\\
&=& \overline{s_{i_1}}^{-1}\cdots\overline{s_{i_n}}^{-1}\bz(t)\overline{s_{i_n}}\cdots\overline{s_{i_1}}\bb,
\end{eqnarray*}
with $\bb\in N$ (because $c(\gamma_1),\ldots,c(\gamma_n)$ are positive roots). 
Again, using the definition of the generalized minors, this implies that
\[
 \Delta_{c^{-1}(\varpi_k),c^{-1}(\varpi_k)}(\bx(t)) = \Delta_{c^{-1}(\varpi_k),c^{-1}(\varpi_k)}(\bz(t)),\qquad (1\le k\le n).
\]
Since for $0\le u < h/2$ we have $c^{u}(\gamma_k) \in \Delta^+(C)$ for all 
$1\le k\le n$, we can iterate
this procedure and show that 
\[
 \Delta_{c^{-u}(\varpi_k),c^{-u}(\varpi_k)}(\bx(t)) = \Delta_{c^{-u}(\varpi_k),c^{-u}(\varpi_k)}(\bz(t)),
 \qquad (1\le k\le n,\ 0\le u < h/2).
\]
So we have proved that the right-hand sides of~\eqref{eq:10.1} 
and~\eqref{eq:10.2} are equal.
This finishes the proof.
\end{proof}

Putting together Proposition~\ref{prop:g-vector} and Theorem~\ref{Thm:11.3}, 
we have now a complete proof of
Theorem~\ref{Thm1}~(c). Moreover, using Corollary~\ref{cor-poly}, this 
finishes also the proof of Theorem~\ref{Thm1}~(a).


\section{Proof of Theorem~\ref{Thm1} {\rm(d)}}

\label{sec:thm1d}

In \cite{FZ2}, when the orientation 
$\Omega$ is such that every vertex of $Q^\circ(C,\Omega)$ is either a source or 
a sink,
the clusters of $\AA$ were described by means of the compatibility degree 
$(\alpha\|\beta)$ of positive roots $\alpha, \beta \in \Delta^+(C)$. 
When $C$ is symmetric, this was extended in \cite{MRZ} to all orientations 
$\Omega$ by introducing
an $\Omega$-compatibility degree $(\alpha\|\beta)_\Omega$ using decorated 
representations
of the Dynkin quiver $Q^\circ(C,\Omega)$. This description was eventually 
extended by Zhu \cite{Z} to all Cartan
matrices $C$ using the cluster category of a modulated graph of the same type 
as $C$ and orientation $\Omega$.

Recall from Section~\ref{subsec:species} the tensor $F$-algebra $T$ attached 
to such a modulated graph,
and 
its indecomposable modules $X(\gamma)$
with dimension vectors $\gamma\in \Delta^+(C)$.
It follows from \cite{Z} that a subset of cluster variables 
$\{x(\gamma_1),\ldots,x(\gamma_k)\}$ of $\AA$ 
(other than $u_1,\ldots , u_n$) is $\Omega$-compatible if and only if for 
every $i,j$ we have 
$\Ext^1_T(X(\gamma_i),X(\gamma_j)) = 0$.

Therefore Theorem~\ref{Thm1}~(d) is now an immediate consequence of 
Theorem~\ref{Thm1}~(c) and
Proposition~\ref{GLS3-Prop5.5}. This finishes the proof of Theorem~\ref{Thm1}.


\section{Examples} \label{sec:examples}

\subsection{Type $B_2$} \label{exB2}

We use the same notation as in \cite[\S13.6]{GLS1}.
Thus $H = H(C,D,\Omega)$ is given by the quiver
\[
\xymatrix{
1 \ar@(ul,ur)^{\vep_1} & 2 \ar[l]^{\alpha_{12}}
}
\]
with unique relation $\vep_1^2 = 0$.  
The skew symmetrizable matrix $B$ is given by
\[
B=
\left(
 \begin{matrix}
  0& 1\\
  -2 & 0
 \end{matrix} 
\right).
 \]
Hence, using the notation of Section~\ref{sect-Fpoly-gvect}, we have 
$z_1 = u_2^{-2}$ and $z_2 = u_1$.

(a) For $M=M(\alpha_1) = E_1$, we have $\Gr_\lf(\br,M)\not = \emptyset$ only if 
$\br=(0,0)$ or $\br=(1,0)$. 
Clearly in both cases $\Gr_\lf(\br,M)$ is just a point, so
$\chi(\Gr_\lf(\br,M))=1$. Hence we have $F_{M}(t_1,t_2) = 1 + t_1$.
Moreover the injective coresolution of $E_1$ is 
$0 \to E_1 \to I_1 \to I_2^2 \to 0$.
Hence $g_{M} = (-1,2)$, and 
\[
 X_{M(\alpha_1)} = X_{E_1} = u_1^{-1}u_2^2(1+z_1)  =\frac{u_2^2+1}{u_1}.
\]

(b) For $M=M(\alpha_2)=E_2=S_2=I_2$, we have $\Gr_\lf(\br,M)\not = \emptyset$ 
only if $\br=(0,0)$ or $\br=(0,1)$. 
Again in both cases 
$\chi(\Gr_\lf(\br,M))=1$. So $F_{M}(t_1,t_2) = 1 + t_2$.
Moreover $M=I_2$ is injective, so $g_{M} = (0,-1)$
\[
 X_{M(\alpha_2)} = X_{E_2} = u_2^{-1}(1+z_2)=\frac{1+u_1}{u_2}.
\]

(c) For $M=M(\alpha_1+\alpha_2) = P_2$, we have $\Gr_\lf(\br,M)\not = \emptyset$
only if $\br=(0,0)$, $\br=(1,0)$, or $\br=(1,1)$. 
Again in all cases $\chi(\Gr_\lf(\br,M))=1$, so $F_{M}(t_1,t_2) = 1 + t_1+t_1t_2$.
The injective coresolution of $P_2$ is $0 \to P_2 \to I_1 \to I_2 \to 0$.
Hence $g_{M} = (-1,1)$ and
\[
 X_{M(\alpha_1+\alpha_2)} = X_{P_2} = u_1^{-1}u_2(1+z_1+z_1z_2) =\frac{u_2^2+1+u_1}{u_1u_2}.
\]

(d) For $M=M(\alpha_1+2\alpha_2) = I_1$, we have 
$\Gr_\lf(\br,M)\not = \emptyset$ only if 
$\br=(0,0)$, $\br=(1,0)$, $\br=(1,1)$ or $\br=(1,2)$. 
If $\br=(1,1)$ then $\Gr_\lf(\br,M))$ is isomorphic to $\PP^1(\C)$ and 
$\chi(\Gr_\lf(\br,M))=2$. In all other cases 
$\chi(\Gr_\lf(\br,M))=1$. So $F_{M}(t_1,t_2) = 1 +t_1 + 2t_1t_2 + t_1t_2^2$.
Since $M=I_1$ is injective, $g_{M}=(-1,0)$, and
\[
 X_{M(\alpha_1+2\alpha_2)} = X_{I_1} = u_1^{-1}(1+z_1+2z_1z_2+z_1z_2^2) = \frac{u_2^2+1+2u_1+u_1^2}{u_1u_2^2}.
\] 

(e) Hence we have
\[
\begin{array}{lll}
 u_1X_{E_1} = 1+u_2^2,&
 u_2X_{E_2} = 1+u_1,&
 u_2X_{P_2} = 1+X_{E_1},\\[3mm]
 u_1X_{I_1} = 1+X_{E_2}^2,&
 X_{P_2}X_{E_2} = 1+X_{I_1}&
 X_{E_1}X_{I_1} = 1+ X_{P_2}^2.
\end{array}
 \]
This shows that the 6 cluster variables of $\AA$ are equal to
\[
 u_1,\ u_2,\ X_{E_1},\ X_{E_2},\ X_{P_2},\ X_{I_1}.
\]

(f) It is easy to see that if $M$ denotes the non rigid locally free 
indecomposable module, we have
$X_M = X_{P_2}$. 

(g) Finally, it is easy to check using the Auslander-Reiten formula and the 
Auslander-Reiten quiver of $H$ displayed
in \cite{GLS1} that the only locally free rigid $H$-modules are of the form
\[
 E_1^a\oplus P_2^b, \quad
 P_2^a\oplus I_1^b, \quad
 I_1^a\oplus E_2^b, \quad
 (a,b \ge 0).
\]
These are in one-to-one correspondence with the cluster monomials of $\AA$ 
which do not contain $u_1, u_2$.
(The clusters containing $u_1$ or $u_2$ are $(u_1,u_2)$, $(u_1,X_{E_2})$, and 
$(u_2,X_{E_1})$.)

\subsection{Type $G_2$}

We use the same notation as in \cite[\S13.9]{GLS1}.
Thus $H = H(C,D,\Omega)$ is given by the quiver
\[
\xymatrix{
1 & 2 \ar[l]^{\alpha_{12}}\ar@(ul,ur)^{\vep_2}
}
\]
with unique relation $\vep_2^3 = 0$.
The skew symmetrizable matrix $B$ is given by
\[
B=
\left(
 \begin{matrix}
  0& 3\\
  -1 & 0
 \end{matrix} 
\right).
 \]
Hence $z_1 = u_2$ and $z_2 = u_1^{-3}$.
  
(a) Easy calculations similar to those in type $B_2$ yield:
\[
 X_{M(\alpha_1)}=X_{E_1} = \frac{u_2+1}{u_1},\quad  X_{M(\alpha_2)} = X_{E_2} = \frac{1+u_1^3}{u_2},\quad 
 X_{M(\alpha_1+\alpha_2)} = X_{I_1} = \frac{u_2+1+u_1^3}{u_1u_2},
\]
\[
 X_{M(3\alpha_1+\alpha_2)} = X_{P_2} = \frac{u_2^3+3u_2^2+3u_2+1+u_1^3}{u_1^3u_2},\quad 
 X_{M(2\alpha_1+\alpha_2)} = \frac{u_2^2+2u_2+1+u_1^3}{u_1^2u_2}. 
\] 

(b) Let $M=M(3\alpha_1+2\alpha_2)$ be the largest indecomposable locally free rigid module.
One calculates
\begin{center}
\begin{tabular}{c|c|c|c|c|c|c|c|}
$\br$ & $(0,0)$ & $(1,0)$ & $(2,0)$ & $(3,0)$ & $(2,1)$ & $(3,1)$ & $(3,2)$ \\
\hline
$\chi(\Gr_\lf(\br,M))$ & 1 & 3 & 3 & 1 & 3 & 2 & 1 \\
\end{tabular}
\end{center}
which yields
\[
  X_{M(3\alpha_1+2\alpha_2)} = \frac{u_2^3+3u_2^2+3u_2+1+ 3u_2u_1^3+ 2u_1^3 + u_1^6}{u_1^3u_2^2}.
\]

(c) As can be seen from the Auslander-Reiten quiver of $H$ (see \cite[Figure 11]{GLS1}), there are two 
non-rigid indecomposable locally free $H$-modules $M_1$ and $M_2$ with rank vector $(3,2)$.
They can be identified via their graded dimension vectors:
\[
M_1 = \bsm 0&0\\1&2\\1&2\\1&2 \esm,\qquad
M_2 = \bsm 1&1\\1&2\\1&2\\0&1 \esm.
\]
Similar calculations show that 
\[
X_{M_1} = X_{M(3\alpha_1+2\alpha_2)},\qquad 
X_{M_2} = X_{M(3\alpha_1+2\alpha_2)}+2.
\]
Thus, $X_{M_2}$ is \emph{not} a cluster variable.

(d) Dually, one can also calculate some PBW basis vectors of the algebra $\cM(H)$.
As mentioned in the introduction, the support of $\theta_\ba$ is in general not reduced
to the orbit of $M_\ba$ in its representation variety 
(see \emph{e.g.} the type $B_2$ case in \cite[Section 8.1]{GLS3}).

Using the calculation of (c), one can check that the orbit of $M_2$ is contained
in the support of $\theta_{3\alpha_1+2\alpha_2}$, and also in the support of 
$\theta_{\alpha_1+\alpha_2}*\theta_{2\alpha_1+\alpha_2}$. This shows that the supports of 
two PBW basis vectors may also intersect non-trivially.

\subsection{Type $B_3$}

We use the same notation as in \cite[\S13.7]{GLS1}.
Here $H = H(C,D,\Omega)$ is given by the quiver
\[
\xymatrix{
1 \ar@(ul,ur)^{\vep_1} & 2 \ar[l]^{\alpha_{12}}\ar@(ul,ur)^{\vep_2} & 3
 \ar[l]^{\alpha_{23}}
}
\]
with relations $\vep_1^2 = \vep_2^2 = 0$, and
$\vep_1\alpha_{12} = \alpha_{12}\vep_2$.
The skew symmetrizable matrix $B$ is given by
\[
\left(
 \begin{matrix}
  0& 1 &0\\
  -1 & 0 & 1\\
  0 & -2 & 0
 \end{matrix} 
\right),
 \]
so $z_1=u_2^{-1}$, $z_2 = u_1u_3^{-2}$, and $z_3 = u_2$.

Let $M=M(\alpha_1+2\alpha_2+2\alpha_3)$. One calculates
\begin{center}
\begin{tabular}{c|c|c|c|c|c|c|c|c}
$\br$ & $(0,0,0)$ & $(1,0,0)$ & $(0,1,0)$ & $(1,1,0)$ & $(1,1,1)$ & $(1,2,0)$ & $(1,2,1)$ & $(1,2,2)$ \\
\hline
$\chi$ & 1 & 1 & 1 & 2 & 2 & 1 & 2 & 1\\
\end{tabular}
\end{center}
which yields
\[
  X_{M(\alpha_1+2\alpha_2+2\alpha_3)} = \frac{u_2u_3^4+u_3^4+3u_1u_2u_3^2+2u_1u_3^2+u_1^2+2u_1^2u_2+u_1^2u_2^2}{u_1u_2^2u_3^2}.
\]
Note that $z_2 = z_1z_2z_3$, so the two rank vectors $(0,1,0)$ and $(1,1,1)$ contribute to the same monomial 
with coefficient $1+2=3$.

\bigskip
{\parindent0cm \bf Acknowledgements.}\,
The authors thank the Mathematisches Forschungsinstitut Oberwolfach
(MFO) for two weeks of hospitality in February/March 2017 where most of this
work was done. 
The second named author thanks Kiyoshi Igusa 
and Gordana Todorov for interesting discussions at the Centre de Recerca 
Matem\'atica (CRM)  Barcelona which arose his interest in this problem.


\end{document}